\documentclass[a4,11pt,fleqn]{article}

\usepackage{amsmath,amssymb,graphicx,amsthm,amsfonts,mathrsfs}
\usepackage{amscd}
\usepackage{slashed}
\usepackage{enumitem}

\usepackage[top=25truemm,bottom=25truemm,left=30truemm,right=30truemm]{geometry}

\theoremstyle{definition}
\newtheorem{Def}{Definition}[section]
\newtheorem{Thm}[Def]{Theorem}
\newtheorem{Prp}[Def]{Proposition}
\newtheorem{Lem}[Def]{Lemma}
\newtheorem{Cor}[Def]{Corollary}

\newtheorem{Remark}[Def]{Remark}

\def\i<#1>{\langle #1 \rangle}

\newcommand{\ra}[1]		{\mbox{$\mathring{\mathrm{#1}}$}}

\newcommand{\mb}[1]	{ \boldsymbol{#1} }

\newcommand{\comp}	{	\circ }
\newcommand{\eps}	{ \varepsilon }

\newcommand{\dirac}	{ \slashed{\partial} }

\newcommand{\deebar} { \overline{\partial} }

\title{An $L^2$-index formula of monopoles with Dirac-type singularities}
\author{Masaki Yoshino}
\date{\today}

\begin{document}
	\maketitle
	\begin{abstract}
		We prove the Fredholmness of Dirac operators of monopoles with Dirac-type singularities on oriented complete Riemannian $3$-folds, and we also calculate the $L^2$-indices of them.
	\end{abstract}
	\section{Introduction}
		Let $(X,g)$ be a complete oriented Riemannian $3$-fold with the bounded scalar curvature.
		Let $Z\subset X$ be a finite subset.
		We fix a spin structure on $X$.
		Let $(V,h)$ be a Hermitian vector bundle on $X\setminus Z$ and $A$ a connection on $(V,h)$.
		Let $\Phi\in\mathrm{End}(V)$ be a skew-Hermitian endomorphism of $(V,h)$.
		The tuple $(V,h,A,\Phi)$ is said to be a monopole on $X\setminus Z$ if the tuple $(V,h,A,\Phi)$ satisfies the Bogomolny equation $F(A)= \ast\nabla_A(\Phi)$,
		where $F(A)$ is the curvature of $A$ and $\ast$ is the Hodge operator.
		Moreover,
		A point $p\in Z$ is a Dirac-type singularity of $(V,h,A,\Phi)$ of weight $\vec{k}_p=(k_{p,i}) \in \mathbb{Z}^{\mathrm{rank}(V)}$ if the monopole $(V,h,A,\Phi)$ satisfies a certain asymptotic behavior around $p\in Z$ (See Definition \ref{Def:def of monopoles} \ref{Enum_2: Def:def of monopoles}.).
		We set the Dirac operators $\dirac^{\pm}_{(A,\Phi)}:\Gamma(X\setminus Z,V\otimes S_X)\to \Gamma(X\setminus Z,V\otimes S_X)$ of $(V,h,A,\Phi)$ to be $\dirac^{\pm}_{(A,\Phi)}(s):= \dirac_{A}(s)\pm \Phi\otimes\mathrm{Id}_{S_X}$,
		where $S_X$ is the spinor bundle on $X$ and $\dirac_A$ is the Dirac operator of $(V,h,A)$.
		We regard $\dirac^{\pm}_{(A,\Phi)}$ as a closed operator $L^2(X\setminus Z,V\otimes S_X) \to L^2(X\setminus Z,V\otimes S_X)$ by considering derivation as a current.
		The main result is the following.
		\begin{Thm}[Theorem \ref{Thm:Index formula for general complete $3$-folds}]
			Let $(V,h,A,\Phi)$ be a monopole of rank $r$ on $X\setminus Z$ such that each $p\in Z$ is a Dirac-type singularity of $(V,h,A,\Phi)$ with weight $\vec{k}_p=(k_{p,i})\in \mathbb{Z}^r$.
			We assume that $(V,h,A,\Phi)$ satisfies the following conditions (\textit{the R\ra{a}de condition}).
			\begin{itemize}
				\item 
					Both $\Phi$ and $F(A)$ are bounded.
				\item
					We have $\nabla_A(\Phi)|_{x} = o(1)$ as $x\to\infty$.
				\item
					There exists a compact religion $Y\supset Z$ such that $Y$ has a smooth boundary $\partial Y$, and the inequality $\mathrm{inf}_{x\in X\setminus Y}\bigl\{|\lambda| \mid \mbox{$\lambda$ is an eigenvalue of $\Phi(x)$}\bigr\} > 0$ is satisfied.
			\end{itemize}
			Then the Dirac operators $\dirac^{\pm}_{(A,\Phi)}$ are Fredholm and adjoint to each other.
			Moreover,
			their indices $\mathrm{Ind}(\dirac^{\pm}_{(A,\Phi)})$ are given as follows:
			\[
				\mathrm{Ind}(\dirac^{\pm}_{(A,\Phi)})
				=
				\mp\left\{
				\sum_{p\in Z}\sum_{k_{p,i}>0} k_{p,i}
				+ \int_{\partial Y}ch(V^{+})
				\right\}
				=
				\pm\left\{
				\sum_{p\in Z}\sum_{k_{p,i}<0} k_{p,i}
				+ \int_{\partial Y}ch(V^{-})
				\right\},
			\]
			where $V^{\pm}$ is a subbundle of $V|_{\partial Y}$ spanned by the eigenvectors of $\mp\sqrt{-1}\Phi$ with positive eigenvalues.
		\end{Thm}
		The celebrated index theorem proved by Atiyah and Singer have been applied in a wide range including gauge theory, differential topology and complex geometry.
		However,
		The indices of elliptic differential operators on odd-dimensional closed manifolds are always $0$.
		Therefore we consider the index theorems of elliptic operators on odd-dimensional open manifolds.
		On one hand,
		Callias \cite{Ref:Cal} proved the index theorem of the Dirac operators of $SU(2)$-bundles on $\mathbb{R}^{2n+1}$ that satisfies a certain asymptotic behavior at infinity.
		Callias's index theorem is generalized to the Dirac operators of vector bundles on odd-dimensional complete spin manifolds by R\ra{a}de \cite{Ref:Rad}.
		On the other hand,
		Kronheimer \cite{Ref:Kro} defined the notion of Dirac-type singularities of monopoles on flat Riemannian $3$-folds,
		and Pauly \cite{Ref:Pau} generalize it to any Riemannian $3$-folds.
		Moreover,
		Pauly proved the index theorem of the deformation complexes on $SU(2)$-monopoles with Dirac-type singularities on closed oriented $3$-folds.
		However,
		Pauly's argument essentially needs the condition $\sum_i k_{p,i}=0$ for any $p\in Z$,
		and it is difficult to apply the argument to calculate the indices of the Dirac operators of $(V,h,A,\Phi)$ even if $X$ is a closed manifold.
		
		The proof of the main result is divided into two parts.
		First we extend Pauly's argument and calculate the indices of $\dirac^{\pm}_{(A,\Phi)}$ when $X$ is a closed manifold (Theorem \ref{Thm:Index of modi. and orig. Dirac} and Corollary \ref{Cor:Index formula of monopole}),
		by constructing a lift of $(V,h,A,\Phi)$ on a $4$-dimensional closed manifold equipped with an $S^1$-action.
		Next we combine our result and R\ra{a}de's index theorem in \cite{Ref:Rad},
		and obtain the index formula on general complete Riemannian $3$-folds (Theorem \ref{Thm:Index formula for general complete $3$-folds}).
		
		This result was obtained in the study of the inverse transform of the Nahm transform from $L^2$-finite instantons on the product of $\mathbb{R}$ and a $3$-dimensional torus $T^3$ to Dirac-type singular monopoles on the dual torus $\hat{T}^3$ of $T^3$ in \cite{Ref:Yos}.
		\subsection*{Acknowledgments}
			I am deeply grateful to my supervisor Takuro Mochizuki for insightful and helpful discussions and advices.
			I thank Tsuyoshi Kato for kindly answering to my question about the $S^1$-equivariant index theorem.
	\section{Preliminary}\label{Sec:Preliminary}
		\subsection{Monopoles with Dirac-type singularities}
			We recall the definition of monopoles with Dirac-type singularities
			by following \cite{Ref:Cha-Hur}.
			\begin{Def}\label{Def:def of monopoles}
				Let $(X,g)$ be an oriented Riemannian 3-fold and $\ast_g$ be the Hodge operator on $X$.
				If there is no risk of confusion,
				then we abbreviate $\ast_g$ to just $\ast$.
				\begin{enumerate}[label=(\roman*)]
					\item
						Let $(V, h)$ be a Hermitian vector bundle with a unitary connection $A$ on $X$.
						Let $\Phi$ be a skew-Hermitian endomorphism of $V$.
						The tuple $(V, h, A, \Phi)$ is said to be a monopole on $X$ if it satisfies the Bogomolny equation $F(A) = \ast \nabla_A(\Phi)$.
					\item\label{Enum_2: Def:def of monopoles}
						Let $Z \subset X$ be a discrete subset.
						Let $(V, h, A, \Phi)$ be a monopole of rank $r \in \mathbb{N}$ on $X \setminus Z$.
						A point $p\in Z$ is called a Dirac-type singularity of the monopole $(V, h, A, \Phi)$ with weight $\vec{k}_p=(k_{p,i}) \in \mathbb{Z}^r$
						if the following holds.
						\begin{itemize}
							\item
								There exists a small neighborhood $B$ of $p$ such that
								$(V,h)|_{B\setminus\{p\}}$ is decomposed into a sum of Hermitian line bundles $\bigoplus_{i=1}^{r} F_{p,i}$
								with $\mathrm{deg}(F_{p,i})=\int_{\partial B} c_1(F_{p,i}) = k_{p,i}$.
							\item
								In the above decomposition, we have the following estimates,
								\begin{align*}
								\left\{
									\begin{array}{l}
										\Phi  = \displaystyle\frac{\sqrt{-1}}{2R_p}\sum_{i=1}^{r} k_{p,i}\cdot Id_{F_{p,i}} + O(1)\\
										\nabla_A(R_p\Phi) = O(1),
									\end{array}
								\right.
								\end{align*}
								where $R_p$ is the distance from $p$.
						\end{itemize}
						For a monopole $(V,h,A,\Phi)$ on $X\setminus Z$,
						if each point $p\in Z$ is a Dirac-type singularity,
						then we call $(V,h,A,\Phi)$ a Dirac-type singular monopole on $(X,Z)$.
				\end{enumerate}
			\end{Def}
			\begin{Remark}
				Let $X$ be a compact manifold and $Z\subset X$ a finite subset.
				For a Dirac-type singular monopole $(V,h,A,\Phi)$ on $(X,Z)$,
				we have $\sum_{p\in Z}\sum_{i}k_{p,i}=0$ by the Stokes theorem,
				where $\vec{k}=(k_{p,i}) \in \mathbb{Z}^{\mathrm{rank}(V)}$ is the weight of $(V,h,A,\Phi)$ at $p\in Z$.
			\end{Remark}
			We also recall the notion of instantons.
			\begin{Def}
				Let $(Y,g)$ be an oriented Riemannian $4$-fold.
				For a Hermitian vector bundle $(V,h)$ on $Y$ and a connection $A$ on $(V,h)$,
				the tuple $(V,h,A)$ is an instanton if the ASD equation $F(A)=-\ast F(A)$ is satisfied.
			\end{Def}
			\begin{Remark}
				If $(Y,g)$ is a K\"{a}hler surface with the K\"{a}hler form $\omega$,
				the condition that a tuple $(V, h, A)$ is an instanton on $Y$ is equivalent to the one that $(V,\deebar_A,h)$ is a holomorphic Hermitian vector bundle satisfying the Hermite-Einstein condition $F(A)\wedge \omega=0$,
				where $\deebar_A$ is the $(0,1)$-part of $\nabla_A$.
			\end{Remark}
			For example,
			we recall the flat Dirac monopole of weight $k\in\mathbb{Z}$.
			Let $g_{i,\mathrm{Euc}}$ denote the canonical metric on $\mathbb{R}^i$.
			For $i\in\mathbb{N}$, we denote by $r_i:\mathbb{R}^i\to \mathbb{R}$ the distance from $0\in\mathbb{R}^i$.
			Let $p:\mathbb{R}^3\setminus\{0\} \to S^2 (\simeq \mathbb{P}^1)$ be the projection.
			Let $\mathcal{O}(k)$ be a holomorphic line bundle on $\mathbb{P}^1$ of degree $k$.
			Let $h_{\mathcal{O}(k)}$ be a Hermitian metric of $\mathcal{O}(k)$ that the Chern connection $A_{\mathcal{O}(k)}$ of $(\mathcal{O}_k,h_{\mathcal{O}(k)})$ has a constant mean curvature.
			Then $(p^{\ast}\mathcal{O}(k),p^{\ast}h_{\mathcal{O}(k)},p^{\ast}A_{\mathcal{O}(k)}, \sqrt{-1}k/2r)$ is a Dirac-type singular monopole on $(\mathbb{R}^3, \{0\})$,
			where $r$ is the distance from the origin.
			We call this monopole the flat Dirac monopole of weight $k$,
			and denote by $(L_k,h_k,A_k,\Phi_k)$.
			
			We will recall the equivalent condition proved by Pauly \cite{Ref:Pau}.
			Let $U\subset \mathbb{R}^3$ be a neighborhood of $0\in\mathbb{R}^3$.
			Let $g$ be a Riemannian metric on $U$.
			We assume that the canonical coordinate of $\mathbb{R}^3$ is a normal coordinate of $g$ at $0$.
			Set the Hopf map $\pi:\mathbb{R}^4=\mathbb{C}^2 \to\mathbb{R}^3=\mathbb{R}\times\mathbb{C}$ to be $\pi(z_1,z_2)=(|z_1|^2-|z_2|^2, 2z_1z_2)$, where we set $z_i=x_i+\sqrt{-1}y_i$.
			We also set the $S^1(= \mathbb{R}/2\pi\mathbb{Z})$-action on $\mathbb{C}^2$ to be $\theta\cdot(z_1,z_2) := (e^{\sqrt{-1}\theta}z_1,e^{-\sqrt{-1}\theta}z_2)$.
			Then the restriction $\pi:\mathbb{R}^4\setminus\{0\} \to \mathbb{R}^3\setminus\{0\}$ forms a principal $S^1$-bundle.
			Then we have $\pi^{\ast}r_3 =r^2_4$.
			\begin{Lem}\label{Lem:prepare of Pauly condition}
				There exist a harmonic function $f:U\setminus\{0\}\to\mathbb{R}$ with respect to the metric $g$ and a $1$-form $\xi$ on $\pi^{-1}(U)$ such that the following hold.
				\begin{itemize}
					\item
						The $1$-form $\omega:=\xi/\pi^{\ast}f$ is a connection of $\pi:\mathbb{R}^4\setminus\{0\} \to \mathbb{R}^3\setminus\{0\}$,
						\textit{i.e.} $\omega$ is $S^1$-invariant, and we have $\omega(\partial_\theta)=1$.
						Here $\partial_\theta$ is the generating vector field of the $S^1$-action on $\mathbb{R}^4\setminus\{0\}$.
					\item
						We have $d\omega=\pi^{\ast}(\ast -df)$.
					\item 
						We have the following estimates:
						\begin{align*}
							\left\{
								\begin{array}{l}
									f = 1/2r_3 + o(1)\\
									\xi = 2(-y_1dx_1 + x_1dy_1 + y_2dx_2 - x_2dy_2 ) + O(r^2_4).\\
								\end{array}
							\right.
						\end{align*}
					\item
						The symmetric tensor $g_4=\pi^{\ast}f(\pi^{\ast}g + \xi^2)$ is a Riemannian metric of $L^2_{5,\mathrm{loc}}$-class on $\pi^{-1}(U)$,
						and we have an estimate $|g_4-2g_{4,\mathrm{Euc}}|_{g_{4,\mathrm{Euc}}}=O(r_4)$.
						Here a function on $\pi^{-1}(U)$ is of $L^2_{k,\mathrm{loc}}$-class if every derivative of $f$ up to order $k$ has a finite $L^2$-norm on any compact subset of $\pi^{-1}(U)$.
				\end{itemize}
			\end{Lem}
			\begin{Prp}[Proposition 5 in \cite{Ref:Pau}]\label{Prp:Pauly condition}
				Let $(V,h,A)$ be a Hermitian vector bundle on $U\setminus\{0\}$ of rank $r$, and $\Phi\in\mathrm{End}(V)$ be a skew-Hermitian endomorphism.
				The tuple $(V,h,A,\Phi)$ is a monopole on $U\setminus\{0\}$ if and only if the tuple $(\pi^{\ast}V,\pi^{\ast}h,\pi^{\ast}A + \xi\otimes \pi^{\ast}\Phi)$ is an instanton on $\pi^{-1}(U)\setminus\{0\}$ with respect to the metric $g_4=\pi^{\ast}f(\pi^{\ast}g + \xi^2)$.
				Moreover,
				$0$ is a Dirac-type singularity of $(V,h,A,\Phi)$ of weight $\vec{k}=(k_i)\in\mathbb{Z}^r$ if and only if the following hold.
				\begin{itemize}
					\item 
						The instanton $(\pi^{\ast}V,\pi^{\ast}h,\pi^{\ast}A - \pi^{\ast}\Phi\otimes\xi)$ can be prolonged over $\pi^{-1}(U)$, 
						and the prolonged connection is represented by an $L^2_{6,\mathrm{loc}}$-valued connection matrix.
						We will denote by $(V_4,h_4,A_4)$ the prolonged instanton.
					\item
						The weight of the $S^1$-action on the fiber $V_4|_{0}$ agrees with $\vec{k}$ up to a suitable permutation.\\
				\end{itemize}
			\end{Prp}
			\begin{Remark}\ \,
				\begin{itemize}
					\item 
						If $g=g_{3,\mathrm{Euc}}$,
						we can choose $f=1/2r_3$ and $\xi = 2(-y_1dx_1 + x_1dy_1 + y_2dx_2 - x_2dy_2 )$.
						Then we have $g_4 = 2g_{4,\mathrm{Euc}}$.
					\item 
						By the Sobolev embedding theorem,
						the connection matrix of $A_4$ is of $C^{3}$ class.
				\end{itemize}
			\end{Remark}
			Let $h_{\mathbb{C}}$ be the canonical Hermitian metric on $\mathbb{C}$.
			We set the Hermitian line bundle $(\tilde{L},\tilde{h}):=(\pi^{-1}(U)\setminus\{0\})\times_{U(1)}(\mathbb{C},h_{\mathbb{C}})$ on $U\setminus\{0\}$ and take the connection $\tilde{A}$ induced by $\omega$.
			Then $(\tilde{L},\tilde{h},\tilde{A},\sqrt{-1}f)$ is a monopole on $U$ with respect to $g$, and $0$ is the Dirac-type singularity of weight $1$.
			We call the monopole $(\tilde{L}_k,\tilde{h}_k,\tilde{A}_k,\sqrt{-1}kf):= (\tilde{L}^{\otimes k},\tilde{h}^{\otimes k},\tilde{A}^{\otimes k},\sqrt{-1}kf)$ a Dirac monopole of weight $k$ with respect to $g$.
			The following proposition is a partial generalization of \cite[Proposition 5.2]{Ref:Moc-Yos}.
			\begin{Prp}
				Let $(V,h,A,\Phi)$ be a monopole on $U\setminus\{0\}$,
				and assume that the point $0$ is a Dirac-type singularity of weight $\vec{k}=(k_i)\in\mathbb{Z}^r$. 
				Then there exist a neighborhood $U'\subset U$ and a unitary isomorphism $\varphi:V|_{U'\setminus\{0\}} \simeq (\bigoplus_{i=1}^r \tilde{L}_{k_i})|_{U'\setminus\{0\}}$ such that the following estimates hold.
				\begin{align*}
					&|A-\varphi^{\ast}(\bigoplus\tilde{A}_{k_i})| =O(1).\\
					&|\Phi - \varphi^{\ast}(\sum \sqrt{-1}k_if\,\mathrm{Id}_{L_{k_i}})| =O(1).
				\end{align*}
			\end{Prp}
			\begin{proof}
				Let $(V',h',A',\Phi')$ be the monopole $\bigoplus_{i=1}^r (\tilde{L}_{k_i},\tilde{h}_{k_i},\tilde{A}_{k_i},\sqrt{-1}k_if)$.
				By Proposition \ref{Prp:Pauly condition},
				the instantons $(\pi^{\ast}V,\pi^{\ast}h,\pi^{\ast}A-\pi^{\ast}\Phi\otimes\xi)$ and$(\pi^{\ast}V',\pi^{\ast}h',\pi^{\ast}A'-\pi^{\ast}\Phi'\otimes\xi)$ can be prolonged over $\pi^{-1}(U)$,
				and denote by $(V_4,h_4,A_4)$ and $(V'_4,h'_4,A'_4)$ respectively.
				Then the weights of $S^1$-actions on the fiber of $V_4$ and $V'_4$ at the origin coincide with each other, and the connections $A_4$ and $A'_4$ are $S^1$-invariant.
				Hence there exist an $S^1$-invariant neighborhood $U'_4\subset \pi^{-1}(U)$ of $0$ and an $S^1$-equivariant unitary isomorphism  $\varphi_4:V_4|_{U'_4} \to V'_4|_{U'_4}$ such that $A_4 - \varphi_4^{\ast}(A'_4)$ vanishes at the origin.
				Hence we have $|A_4 - \varphi_4^{\ast}(A'_4)| = O(r_4)$.
				Since $f = 1/2r_3 + o(1)$ and $\xi$ is orthogonal to $\pi^{\ast}(T^{\ast}\mathbb{R}^3)$ with the metric $g_4=\pi^{\ast}f(\pi^{\ast}g + \xi^2)$,
				the unitary isomorphism $\varphi:V|_{U'\setminus\{0\}}\to V'|_{U'\setminus\{0\}}$ induced by $\varphi_4$ satisfies the desired estimates,
				where we put $U' := \pi(U'_4)$.
			\end{proof}
			By the estimates in Lemma \ref{Lem:prepare of Pauly condition},
			we also obtain the following approximation.
			\begin{Cor}\label{Cor:Appro. by Dirac monopole}
				Let $(V,h,A,\Phi)$ be a monopole on $U\setminus\{0\}$,
				and assume that the point $0$ is a Dirac-type singularity of weight $\vec{k}=(k_i)\in\mathbb{Z}^r$. 
				Then there exist a neighborhood $U'\subset U$ and a unitary isomorphism $\varphi:V|_{U'} \simeq (\bigoplus_{i=1}^r L_{k_i})|_{U'}$ such that the following estimates hold.
				\begin{align*}
					&|A-\varphi^{\ast}(\bigoplus A_{k_i})| =O(1).\\
					&|\Phi - \varphi^{\ast}(\frac{\sqrt{-1}}{2r_3}\sum k_i \mathrm{Id}_{L_{k_i}})| =O(1).
				\end{align*}
			\end{Cor}
		\subsection{Local properties of harmonic spinors of the flat Dirac monopoles}\label{subsec:Loc. prop. of flat Dirac mon.}
			Let $(X,g)$ be an $n$-dimensional oriented spin manifold with a fixed spin structure.
			We denote by $S_X$ the spinor bundle on $X$,
			and by $\mathrm{clif}:T^{\ast}X\to \mathrm{End}(S_X)$ the Clifford product.
			If $n$ is an odd number,
			then we assume $(\sqrt{-1})^{(n+1)/2}\mathrm{clif}(\mathrm{vol}_{(X,g)})= -\mathrm{Id}_{S_X}$,
			where we use the canonical linear isomorphism between the exterior algebra and the Clifford algebra.
			The spinor bundle $S_X$ has the induced connection $A_{S_X}$ by the Levi-Civita connection on $X$,
			and we set the Dirac operator $\dirac_X:\Gamma(X,S_X)\to\Gamma(X,S_X)$ to be $\dirac_X(f):=\mathrm{clif}\comp \nabla_{A_{S_X}}(f)$.
			For a vector bundle $(V,h)$ on $X$ and a connection $A$ on $(V,h)$,
			we also set  the Dirac operator $\dirac_{A}:\Gamma(X,S_X\otimes V)\to\Gamma(X,S_X\otimes V)$ to be $\dirac_A(s):=\mathrm{clif}\comp \nabla_{A_{S_X}\otimes A}(s)$.
			If $n$ is even,
			then we have the decomposition $S_X = S^{+}_{X}\oplus S^{-}_{X}$, and the Dirac operator $\dirac_A$ is also decomposed into sum of the positive and negative Dirac operators $\dirac^{\pm}_{A}:\Gamma(X,S^{\pm}_{X}\otimes V)\to\Gamma(X,S^{\mp}_{X}\otimes V)$.
			If $\mathrm{dim}(X)=3$,
			then for a monopole $(V,h,A,\Phi)$ on $X$ we also set the Dirac operators $\dirac^{\pm}_{(A,\Phi)}:\Gamma(X,V\otimes S_X)\to\Gamma(X,V\otimes S_X)$ to be $\dirac^{\pm}_{(A,\Phi)}(f):=\dirac_{A}(f) \pm (\Phi\otimes\mathrm{Id}_{S_X})(f)$.
			
			For a differential operator $P:\Gamma(X,V_1)\to\Gamma(X,V_2)$ between Hermitian vector bundles $(V_1,h_1)$ and $(V_2,h_2)$ on $X$,
			we regard $P$ as the closed operator $P:L^2(X,V_1)\to L^2(X,V_2)$ with the domain $\mathrm{Dom}(P) := \{s\in L^2(X,V_1) \mid P(s) \in L^2\}$,
			where $P(s)$ is the derivative as a current.
			We regard $\mathrm{Dom}(P)$ as a Banach space equipped with the graph norm $||s||_{P} := ||s||_{L^2} +||P(s)||_{L^2}$.
			\begin{Remark}
				Any $3$-dimensional oriented manifolds are parallelizable,
				and hence they have spin structures.
			\end{Remark}
			
			Let $S_{\mathbb{R}^3}$ be the spinor bundle on $\mathbb{R}^3$ with respect to the trivial spin structure,
			and $d$ be the trivial connection on $S_{\mathbb{R}^3}$.
			By using the projection $p:\mathbb{R}^3\setminus\{0\}\to S^2$,
			We combine the Dirac operators of the Dirac monopole $(L_k,h_k,A_k,\Phi_k)=(p^{\ast}\mathcal{O}(k),p^{\ast}h_{\mathcal{O}(k)},p^{\ast}A_{\mathcal{O}(k)} \sqrt{-1}k/2r)$ with the Dirac operators of $\mathcal{O}(k)$ on $\mathbb{P}^1=S^2$.
			Let $S_{S^2}=S^{+}_{S^2}\oplus S^{-}_{S^2}$ be the spinor bundle on $(S^2,g_{S^2})$, and $\dirac^{\pm}_{S^2}:\Gamma(S^2,S^{\pm}_{S^2})\to\Gamma(S^2,S^{\mp}_{S^2})$ the Dirac operators on $S^2$.
			By the isometry $\mathbb{R}^3\simeq \left(\mathbb{R}_{+}\times S^2,(dr_3)^2 + r_3^2g_{S^2}\right)$ we obtain the unitary isomorphisms $S_{\mathbb{R}^3}|_{\mathbb{R}^3\setminus\{0\}} \simeq p^{\ast}S_{S^2}$.
			According to Nakajima \cite{Ref:Nak},
			under the identification $S_{\mathbb{R}^3}|_{\mathbb{R}^3\setminus\{0\}} \simeq p^{\ast}S_{S^2}$ the Dirac operator $\dirac_{\mathbb{R}^3}$ on $\mathbb{R}^3\setminus\{0\}$ is written as follows:
			\begin{align*}
				\dirac_{\mathbb{R}^3} 
					&= 
					\frac{1}{r_3}\left(\begin{array}{cc}
						\displaystyle\sqrt{-1}( r_3\frac{\partial}{\partial r_3} + 1) &
						\dirac^{-}_{S^2}\\
						\\
						\dirac^{+}_{S^2} &
						\displaystyle -\sqrt{-1}( r_3\frac{\partial}{\partial r_3} + 1)
					\end{array}\right).
			\end{align*}
			Therefore we obtain the following equality.
			\[
				\dirac^{\pm}_{(A_k,\Phi_k)} = 
				\frac{1}{r_3}\left(\begin{array}{cc}
					\displaystyle\sqrt{-1}( r_3\frac{\partial}{\partial r_3} + \frac{2\pm k}{2}) & \dirac^{-}_{\mathcal{O}(k)}\\
					\\
					\dirac^{+}_{\mathcal{O}(k)} & \displaystyle -\sqrt{-1}( r_3\frac{\partial}{\partial r_3} + \frac{2\mp k}{2})
				\end{array}\right).
			\]
			
			By the isomorphisms $S^{+}_{S^2}\simeq \Omega^{0,0}(\mathcal{O}(-1))$, $S^{-}_{S^2}\simeq \Omega^{0,1}(\mathcal{O}(-1))$ and $\dirac_{S^2} = \dirac^{+}_{S^2}+\dirac^{-}_{S^2}= \sqrt{2}(\deebar_{\mathcal{O}(-1)} + \deebar_{\mathcal{O}(-1)}^{\bigstar})$,
			we obtain $\mathrm{Ker}(\dirac^{+}_{\mathcal{O}(k)}) \simeq H^0(\mathbb{P}^1,\mathcal{O}(k-1))$ and $\mathrm{Ker}(\dirac^{-}_{\mathcal{O}(k)}) \simeq H^1(\mathbb{P}^1,\mathcal{O}(k-1))$,
			where $\deebar_{\mathcal{O}(-1)}^{\bigstar}$ is the formal adjoint of $\deebar_{\mathcal{O}(-1)}$.
			Let $f^{\pm}_{\nu} \in L^2(S^2,S^{\pm}_{S^2}\otimes \mathcal{O}(k))\;(\nu\in\mathbb{N})$ be the all eigenvectors of the operators $\dirac^{-}_{\mathcal{O}(k)}\comp \dirac^{+}_{\mathcal{O}(k)}$ and $\dirac^{+}_{\mathcal{O}(k)}\comp \dirac^{-}_{\mathcal{O}(k)}$ with non-zero eigenvalues respectively.
			We set $n_{\nu} >0$ to be the eigenvalue of $f^{\pm}_{\nu}$.
			Then, According to \cite{Ref:Alm-Pri},
			we have $\{n_{\nu}\} = \{q^2 + |k|q \;;\; q\in\mathbb{N}\}$.
			We set $q_{\nu}>0$ to satisfy $n_{\nu}= q_{\nu}^2 + |k|q_{\nu}$.
			We may assume that $\{f^{\pm}_{\nu}\}$ forms an orthonormal system and satisfies the relations $\dirac^{\pm}_{\mathcal{O}(k)}(f^{\pm}_{\nu})=\sqrt{n_{\nu}} f^{\mp}_{\nu}$ for any $\nu\in\mathbb{N}$.
			By the elliptic inequality and the Sobolev inequality,
			there exist $C',C''>0$ such that $||f^{\pm}_{\nu}||_{L^6} < C''||f^{\pm}_{\nu}||_{L^2_1} \leq C'(||f^{\pm}_{\nu}||_{L^2}+ ||\dirac^{\pm}_{\mathcal{O}(k)}(f^{\mp}_{\nu})||_{L^2}) = C'(1+\sqrt{n_\nu})$.
			Then by the interpolation inequality we obtain $||f^{\pm}_{\nu}||_{L^3} \leq (||f^{\pm}_{\nu}||_{L^2})^{1/2}\cdot(||f^{\pm}_{\nu}||_{L^6})^{1/2} = C\sqrt{1+\sqrt{n_\nu}}$,
			where we put $C:=\sqrt{C'}$.
			Hence we obtain the following lemma.
			\begin{Lem}\label{Lem:Est. of Harm. sec}
				We have the estimate $||f^{\pm}_{\nu}||_{L^3} = O(\sqrt{q_{\nu}})$.
			\end{Lem}
			
			Through the above arguments,
			we obtain the following proposition.
			\begin{Prp}\label{Prp:Ker of Dirac monopole}
				Let $s$ be a section of $L_k\otimes S_{\mathbb{R}^3}$ on a punctured ball $B(r)^{\ast} := \{x\in\mathbb{R}^3 \mid 0<|x|<r \}$ for some $r>0$.
				\begin{enumerate}[label=(\roman*)]
					\item
						If we have $s\in L^2(B(r)^{\ast},L_k\otimes S_{\mathbb{R}^3}) \cap \mathrm{Ker}(\dirac^{+}_{(A_k,\Phi_k)})$,
						then there exists a sequence $\{c_{\nu}\} \subset \mathbb{C}$ such that we have 
						\[
						s = \sum_{\nu\in\mathbb{N}} c_{\nu} \left(a^{+}_{\nu}(r)f^{+}_{\nu} + a^{-}_{\nu}(r)f^{-}_{\nu}\right).
						\]
						Here the functions $a^{\pm}_{\nu}$ are given as follows:
						\begin{align*}
							&a^{+}_{\nu}(r) = r^{-1+q_{\nu}+|k|/2}.\\
							&a^{-}_{\nu}(r) = \frac{q_{\nu}+\mathrm{max}(0,k)}{\sqrt{-1}\sqrt{q_{\nu}^2+|k|q_{\nu}}}\,r^{-1+q_{\nu}+|k|/2}.
						\end{align*}
					\item 
						If we have $s\in L^2(B(r)^{\ast},L_k\otimes S_{\mathbb{R}^3}) \cap \mathrm{Ker}(\dirac^{-}_{(A_k,\Phi_k)})$,
						then there exist a sequence $\{c_{\nu}\} \subset \mathbb{C}$, $\alpha^{+} \in \mathrm{Ker}(\dirac^{+}_{\mathcal{O}(k)}) =H^0(\mathbb{P}^1,\mathcal{O}(k-1))$ and $\alpha^{-} \in \mathrm{Ker}(\dirac^{-}_{\mathcal{O}(k)})=H^1(\mathbb{P}^1,\mathcal{O}(k-1))$ such that we have 
						\[
							s = \sum_{+,-}\alpha^{\pm}\rho^{\pm}(r) +  \sum_{\nu\in\mathbb{N}} c_{\nu} \left(b^{+}_{\nu}(r)f^{+}_{\nu} + b^{-}_{\nu}(r)f^{-}_{\nu}\right).
						\]
						Here the functions $\rho^{\pm}$ and $b^{\pm}_{\nu}$ are given as follows:
						\begin{align*}
							&\rho^{\pm}(r) = r^{-1\pm k/2}.\\
							&b^{+}_{\nu}(r) = r^{-1+q_{\nu}+|k|/2}.\\
							&b^{-}_{\nu}(r) = \frac{q_{\nu}+\mathrm{max}(0,-k)}{\sqrt{-1}\sqrt{q_{\nu}^2+|k|q_{\nu}}}\,r^{-1+q_{\nu}+|k|/2}.
						\end{align*}
				\end{enumerate}
			\end{Prp}
			By the above proposition,
			we obtain the following corollary.
			\begin{Cor}\label{Cor:Cptness of Kernel}
				For arbitrary positive numbers $r>r'>0$,
				the restriction map $L^2(B(r)^{\ast},L_k\otimes S_{\mathbb{R}^3})\cap\mathrm{Ker}(\dirac^{\pm}_{(A_k,\Phi_k)}) \to L^2(B(r')^{\ast},L_k\otimes S_{\mathbb{R}^3})$ is a compact map.
			\end{Cor}
			As a preparation of Proposition \ref{Prp:Dol. lemma for Dirac monopoles},
			we prove the following lemma.
			\begin{Lem}\label{Lem:est. of ode.}
				Let $t_0>0$ be a positive number and $\alpha$ a real number.
				Set the constant $C_{\alpha}$ is given by
				\[
				C_{\alpha}=
				\left\{\begin{array}{ll}
				|2\alpha-1|^{-1/2} & (\alpha\neq 1/2)\\
				1 & (\alpha=1/2).
				\end{array}\right.
				\]
				There exists a compact operator $K_{\alpha}:L^2(0,t_0) \to C^0([0,t_0])$ such that for any $f\in L^2(0,t_0)$,
				the function $g:=K_{\alpha}(f)$ satisfies the estimate $|g(t)|\leq C_{\alpha}||f||_{L^2}\cdot t^{1/2}(1+\log(t_0/t)^{1/2})\leq C_{\alpha}||f||_{L^2}\cdot \sqrt{t_0}(1+1/\sqrt{e})$ and the differential equation $t\partial_t(g/t) + \alpha(g/t) = f$,
				where $C^0([0,t_0])$ is the Banach space consisting of bounded continuous functions on $[0,t_0]$.
			\end{Lem}
			\begin{proof}
				We set $g=K_{\alpha}(f)$ to be
				\begin{align*}
					g(t):=
					\left\{\begin{array}{ll}
						\displaystyle t^{-\alpha+1}\int_{0}^{t} f(x)x^{\alpha-1}dx & (\alpha> 1/2)\\
						\\
						\displaystyle -t^{-\alpha+1}\int_{t}^{t_0} f(x)x^{\alpha-1}dx & (\alpha\leq 1/2).
					\end{array}\right.
				\end{align*}
				Then, by a direct calculation we have $t\partial_t(g/t) + \alpha(g/t) = f$.
				If $\alpha\neq 1/2$,
				then we obtain $|g(t)|\leq t^{-\alpha+1}||f||_{L^2} \sqrt{t^{2\alpha-1}/|2\alpha-1|}=|2\alpha-1|^{-1/2}||g||_{L^2}\cdot  t^{1/2}$.
				If $\alpha=1/2$,
				then we have $|g(t)|\leq ||f||_{L^2}\cdot t^{1/2}\log(t_0/t)^{1/2}$.
				As a consequence of the above inequalities,
				we obtain the desired estimate.
				By this estimate,
				the compactness of $K_{\alpha}$ follows from the Ascoli-Arzel\`{a} theorem and the differential equation.
			\end{proof}
			\begin{Prp}\label{Prp:Dol. lemma for Dirac monopoles}
				Let $r>0$ be a positive number.
				There exists a compact map $G^{\pm}:L^2(B(r)^{\ast},L_k\otimes S_{\mathbb{R}^3}) \to L^2(B(r)^{\ast},L_k\otimes S_{\mathbb{R}^3})$ such that we have $R(G^{\pm}) \subset \mathrm{Dom}(\dirac^{\pm}_{(A_k,\Phi_k)})$ and $\dirac^{\pm}_{(A_k,\Phi_k)} \comp G^{\pm} = \mathrm{Id}$.
				Moreover,
				we have $R(G^{\pm}) \subset L^3(B(r)^{\ast},L_k\otimes S_{\mathbb{R}^3})$, and hence $G^{\pm}:L^2\to L^3$ is bounded.
			\end{Prp}
			\begin{proof}
				The proof for $\dirac^{+}_{(A_k,\Phi_k)}$ remains valid for $\dirac^{-}_{(A_k,\Phi_k)}$.
				Hence we prove only for $\dirac^{+}_{(A_k,\Phi_k)}$.
				The subspace that is spanned by $\mathrm{Ker}(\dirac^{\pm}_{\mathcal{O}(k)})$ and $\{f^{\pm}_\nu\}$ is dense in $L^2(S^2,S_{S^2}\otimes \mathcal{O}(k))$.
				Hence for any $s\in L^2(B(r)^{\ast},L_k\otimes S_{\mathbb{R}^3})$ there exist measurable maps $\alpha^{\pm}:(0,r) \to \mathrm{Ker}(\dirac^{\pm}_{\mathcal{O}(k)})$ and $s^{\pm}_{\nu}:(0,r)\to \mathbb{C}$ such that we have
				\[
					s = \alpha^{+} + \alpha^{-} + \sum_{\nu}\left(s^{+}_{\nu}f^{+}_{\nu} + s^{-}_{\nu}f^{-}_{\nu}\right)
				\]
				and 
				\[
					||s||^2_{L^2} = ||r_3\alpha^{+}||^2_{L^2}+||r_3\alpha^{-}||^2_{L^2} + \sum_{\nu}\left(||r_3s^{+}_{\nu}||^2_{L^2} + ||r_3s^{-}_{\nu}||^2_{L^2}\right).
				\]
				By some linear-algebraic operations and Lemma \ref{Lem:est. of ode.},
				we can take an element $t= \beta^{+} + \beta^{-} + \sum_{\nu}\left(t^{+}_{\nu}f^{+}_{\nu} + t^{-}_{\nu}f^{-}_{\nu}\right) \in L^2(B(r)^{\ast},L_k\otimes S_{\mathbb{R}^3})$ such that we have $\dirac^{+}_{(A_k,\Phi_k)}(t)=s$ and 
				\begin{align*}
					||t||^2_{L^2} 
					&= ||r_3\beta^{+}||^2_{L^2} + ||r_3\beta^{-}||^2_{L^2} + \sum_{\nu}\left(||r_3t^{+}_{\nu}||^2_{L^2} + ||rt^{-}_{\nu}||^2_{L^2}\right)\\
					&\leq ||r_3\alpha^{+}||_{L^2}^2 + ||r_3\alpha^{-}||_{L^2}^2 \\
					&\ \ \ + \sum_{\nu}\left\{\mathrm{max}(C_{1+(2q_{\nu}+k)/2},C_{1-(2q_{\nu}+k)/2})^2\left(||r_3s^{+}_{\nu}||^2_{L^2} + ||r_3s^{-}_{\nu}||^2_{L^2}\right)\right\},
				\end{align*}
				where $C_{\alpha}$ is the constant in Lemma \ref{Lem:est. of ode.}.
				Then We set $G^{+}(s):=t$, and $G^{+}$ is linear because all constructions of $G^{+}$ are linear.
				Since $C_{1\pm(2q_{\nu}+k)/2}=o(1)\;(\nu\to\infty)$,
				the compactness of $G^{+}$ is deduced from the compactness of $K_{\alpha}$ in Lemma \ref{Lem:est. of ode.}.
				
				By the definition we have $2\sqrt{q_{\nu}}\cdot C_{1\pm(2q_{\nu}+k)/2}\to 1\;\;(\nu\to\infty)$.
				Hence $||f^{\pm}_{\nu}||_{L^3}\cdot C_{1\pm(2q_{\nu}+k)/2} = O(1)$ by Lemma \ref{Lem:Est. of Harm. sec}.
				Therefore we obtain $||t||_{L^3}< \infty$ and the proof is complete.
			\end{proof}
			\begin{Cor}\label{Cor:Relich type thm for Dirac monopole}
				For any positive numbers $r>r'>0$,
				the restriction map $L^2(B(r)^{\ast},L_k\otimes S_{\mathbb{R}^3})\cap\mathrm{Dom}(\dirac^{\pm}_{(A_k,\Phi_k)}) \to L^2(B(r')^{\ast},L_k\otimes S_{\mathbb{R}^3})$ is a compact operator.
			\end{Cor}
			\begin{proof}
				Let $\{f_n\}$ be a bounded sequence in $L^2(B(r)^{\ast},L_k\otimes S_{\mathbb{R}^3})\cap\mathrm{Dom}(\dirac^{\pm}_{(A_k,\Phi_k)})$.
				By using $G^{\pm}$ in Proposition \ref{Prp:Dol. lemma for Dirac monopoles},
				we set $\tilde{f}_n := G^{\pm}(\dirac^{\pm}_{(A_k,\Phi_k)}(f_n))$.
				Since $G^{\pm}$ is compact,
				there exists a subsequence $\{f_{n_k}\}$ such that $\{\tilde{f}_{n_k}\}$ is convergent.
				Hence we may assume that $\{\tilde{f}_n\}$ is convergent.
				Then we have $\dirac^{\pm}_{(A_k,\Phi_k)}(f_n-\tilde{f}_n)=0$.
				By Corollary \ref{Cor:Cptness of Kernel},
				$\{(f_n-\tilde{f}_n)|_{B(r')}\}$ has a convergent subsequence.
				Therefore $\{f_n|_{B(r')}\}$ also has a convergent subsequence.
			\end{proof}
		\subsection{A local lift of the Dirac operators of the flat Dirac monopoles}
			Let $k\in\mathbb{Z}$ be an integer.
			For the flat Dirac monopole $(V,h,A,\Phi):=(L_k,h_k,A_k,\sqrt{-1}k/2r_3)$ on $(\mathbb{R}^3,\{0\})$,
			we denote by $(V_4,h_4,A_4)$ the prolongation of the instanton $(\pi^{\ast}V,\pi^{\ast}h,\pi^{\ast}A - \xi\otimes\pi^{\ast}\Phi)$ over $\mathbb{R}^4$,
			where $\xi= 2\{(x_1dy_1-y_1dx_1) - (x_2dy_2-y_2dx_2)\}$.
			We compare the Dirac operators $\dirac^{\pm}_{(A,\Phi)}$ and $\dirac^{\pm}_{A_4}$.
			
			We denote by $X$ and $P$ the punctured spaces $\mathbb{R}^3\setminus\{0\}$ and $\mathbb{R}^4\setminus\{0\}$ respectively.
			Set the function $f:\mathbb{R}^3\setminus\{0\}\to \mathbb{R}_+$ to be $f(t,x,y):=1/2r_3$.
			We also set $g_P := 2g_{4,\mathrm{Euc}}$.
			Since $g_P = 2g_{4,\mathrm{Euc}}= \pi^{\ast}f(\pi^{\ast}g + \xi^2)$,
			we have the orthogonal decomposition $TP \simeq \mathbb{R}\partial_{\theta}\oplus \pi^{\ast}TX$.
			Let $\mathscr{S}$ be the spin structure of $\mathbb{R}^3$\;\textit{i.e.}\;$\mathscr{S}$ is a principal $Spin(3)$-bundle on $\mathbb{R}^3$ that satisfies $\mathscr{S}\times_{Spin(3)}(\mathbb{R}^3,g_{3,Euc}) \simeq T\mathbb{R}^3$.
			Let $\rho:Spin(3)\to Spin(4)$ be the lift of the homomorphism $SO(3)\to SO(4)$ which is induced by $\mathbb{R}^3\ni p \to (0,p)\in\mathbb{R}^4$.
			We set $\mathscr{S}_4 := \pi^{\ast}(\mathscr{S})\times_{\rho}Spin(4)$.
			Then we have $\mathscr{S}_4 \times_{Spin(4)} (\mathbb{R}^4\setminus\{0\}) \simeq (P\times \mathbb{R}) \oplus \pi^{\ast}TX$,
			and hence $\mathscr{S}_4$ is a spin structure on $P$.
			Under the isomorphisms $Spin(3)\simeq SU(2)$ and $Spin(4)\simeq SU(2)_{+}\times SU(2)_{-}$,
			the homomorphism $\rho$ is written as $\rho(g)=(g,g)$.
			Therefore we have the following proposition.
			\begin{Prp}\label{Prp:local lift of clif.}
				The following claims are satisfied.
				\begin{itemize}
					\item 
						We have the unitary isomorphisms $\pi^{\ast}S_X \simeq S^{\pm}_{P}$.
					\item 
						Under the above isomorphisms,
						the Clifford product on $P$ can be represented as follows:
						\begin{alignat*}{3}
							&\mathrm{clif}_{P}(\xi)&=&
								(\pi^{\ast}f)^{-1/2}\left(
								\begin{array}{cc}
								0 & \mathrm{Id}\\
								-\mathrm{Id} & 0
								\end{array}
								\right).&&\\
							&\mathrm{clif}_{P}(\pi^{\ast}\alpha)\ &=&
								(\pi^{\ast}f)^{-1/2}\left(
								\begin{array}{cc}
								0 & \mathrm{clif}_{X}(\alpha)\\
								\mathrm{clif}_{X}(\alpha) & 0
								\end{array}
								\right)&&\ \ \ \ \ (\alpha\in\Gamma(X,\Omega^1(X))).\\
						\end{alignat*}
				\end{itemize}
			\end{Prp}
			Since the isomorphisms $\pi^{\ast}S_X \simeq S^{\pm}_{P}$ are unitary,
			we have $||\pi^{\ast}s||^2_{L^2(P)} = \int_{P} |\pi^{\ast}s|^2 (-\pi^{\ast}f^2\cdot\xi\wedge \pi^{\ast}d\mathrm{vol}_{X})= 2\pi||f^{1/2}s||^2_{L^2(X)}$ for any $s\in \Gamma(X,V\otimes S_X)$.
			Hence the operator $\pi^{\dagger}(s) := \pi^{\ast}((2\pi f)^{-1/2}\,s)$ are isometric isomorphisms between $L^2(X,V\otimes S_X)$ and $L^2(P,V_4\otimes S^{\pm}_P)^{S^1}$.
			
			On one hand,
			we take a global flat unitary frame $\mb{e}^3=(e^{3}_1,e^{3}_2)$ of $S_X$ that satisfies the following.
			\begin{alignat*}{2}
				&\mathrm{clif}_{X}(dt)\mb{e}^3&=&\mb{e}^3
					\left(
					\begin{array}{cc}
					\sqrt{-1} & 0\\
					0 & -\sqrt{-1}
					\end{array}
					\right).\\
				&\mathrm{clif}_{X}(dx)\mb{e}^3&=&\mb{e}^3
					\left(
					\begin{array}{cc}
					0 & -1\\
					1 & 0
					\end{array}
					\right).\\
				&\mathrm{clif}_{X}(dt)\mb{e}^3&=&\mb{e}^3
					\left(
					\begin{array}{cc}
					0 & \sqrt{-1}\\
					\sqrt{-1} & 0
					\end{array}
					\right).\\
			\end{alignat*}
			On the other hand,
			we have the isomorphisms $S^{+}_{P}\simeq \Omega^{0,0}_{\mathbb{C}^2} \oplus \Omega^{0,2}_{\mathbb{C}^2}$ and $S^{-}_{P} \simeq \Omega^{0,1}_{\mathbb{C}^2}$.
			Moreover,
			under the isomorphisms we also have $\dirac_{A_4} = \sqrt{2}(\deebar_{A_4} + \deebar_{A_4}^{\bigstar})$ and $\mathrm{clif}_{P}(\alpha) = \sqrt{2}( \alpha^{(0,1)}\wedge - \lrcorner(\alpha^{(1,0)})^{\flat})$ for a $1$-form $\alpha$ on $P$,
			where $\lrcorner$ means the interior product and $(\alpha)^{\flat}$ is the image of $\alpha$ under the isomorphism $\Omega^{1,0}_{\mathbb{C}^2}\simeq T^{(0,1)}\mathbb{C}^2$ induced by the metric $g_P$.
			Here we set $S^1$-invariant global unitary frames $\mb{e}^{\pm}=(e^{\pm}_{1},e^{\pm}_{2})$ of $S^{\pm}_P$ to be the following.
			\begin{alignat*}{2}
				&e^{+}_1 &:=&
					1.\\
				&e^{+}_2 &:=& 
					-\left(\pi^{\ast}(-\xi)^{0,1}/|\pi^{\ast}(-\xi)^{0,1}|\right)\wedge \left(\pi^{\ast}(d\bar{z})/|\pi^{\ast}(d\bar{z})|\right).\\
				&e^{-}_1 &:=&
					\;\pi^{\ast}(-\xi)^{0,1}/|\pi^{\ast}(-\xi)^{0,1}|.\\
				&e^{+}_2 &:=&
					\;\pi^{\ast}(d\bar{z})/|\pi^{\ast}(d\bar{z})|.
			\end{alignat*}
			Then,
			with respect to the frames $\mb{e}^{\pm}$ and $\mb{e}^3$,
			the representations of Clifford products of $X$ and $P$ coincide as in the sense of Proposition \ref{Prp:local lift of clif.}.
			Therefore we may assume $\pi^{\ast}\mb{e}^3 = \mb{e}^{\pm}$.
			Hence by a direct calculation we obtain the following proposition.
			\begin{Prp}\label{Prp:Local lift of Dirac op.}
				For the flat Dirac monopole $(V,h,A,\Phi)$,
				the equalities
				\[
					\pi^{\dagger}\comp\left(\dirac^{+}_{(A,\Phi)}f^{-1/2}\right)(s)
					=
					\dirac^{+}_{A_4} \comp \pi^{\dagger}(s)
				\]
				and
				\[
					\pi^{\dagger}\comp\left(f^{-1/2}\dirac^{-}_{(A,\Phi)}\right)(s)
					=
					\dirac^{+}_{A_4} \comp \pi^{\dagger}(s)
				\]
				are satisfied for any $s\in\Gamma(X,V\otimes S_X)$.
			\end{Prp}

\section{An index formula of Dirac operators on compact $3$-folds}\label{Sec:CompactCase}
		Let $(X,g)$ be a closed oriented spin $3$-fold and $Z$ a finite subset.
		Let $\mathscr{S}$ be a spin structure on $(X,g)$ \,\textit{i.e.} $\mathscr{S}$ is a principal $Spin(3)$-bundle on $X$ such that $\mathscr{S}\times_{Spin(3)}(\mathbb{R}^3,g_{3,\mathrm{Euc}}) \simeq (TX,g)$.
		Let $(V,h,A,\Phi)$ be a Dirac-type singular monopole on $(X,Z)$ of rank $r$,
		and we denote by $\vec{k}_p=(k_{p,i}) \in \mathbb{Z}^r$ the weight of $(V,h,A,\Phi)$ at each $p\in Z$.
		\subsection{Fredholmness of Dirac operators}
			For a sufficiently small $\eps>0$,
			we set $B(Z,\eps):= \coprod_{p\in Z}B(p,\eps) = \coprod_{p\in Z} \{x\in X\mid d_g(x,p)<\eps\}$,
			where $d_g:X\times X \to \mathbb{R}$ is the distance function with respect to $g$.
			Let $(x_p^1,x_p^2,x_p^3)$ be a normal coordinate at $p$ on $B(p,\eps)$,
			and set the flat metric $g'$ on $B(Z,\eps)$ to be $g'|_{B(p,\eps)} := \sum_i (dx_p^i)^2$.
			We take a smooth bump function $\rho:X\to[0,1]$ satisfying $\rho(B(Z,\eps/2))=1$ and $\rho(X \setminus B(Z,3\eps/4))=0$,
			and set a metric $\tilde{g}:=(1-\rho)g + \rho\cdot g'$.
			Then we have $g|_{X\setminus B(Z,\eps)} = \tilde{g}|_{X\setminus B(Z,\eps)}$ and $|g-\tilde{g}|_{g}= O(R_p^2)$ on $B(p,\eps)$ for any $p\in Z$,
			where $R_p$ is the distance from $p$.
			Hence there exists an isometric isomorphism $\mu:(TX,g)\simeq (TM,\tilde{g})$ such that $\mu|_{X\setminus B(Z,\eps)} = \mathrm{Id}_{TM}$ and $|\mu-\mathrm{Id}_{TM}|_{g}= O(R_p^2)$ on $B(p,\eps)$ for any $p\in Z$.
			Therefore we obtain the following lemma.
			\begin{Lem}\label{Lem:Pertubation of Clif prod.}
				For a $1$-form $\alpha$,
				we have an equality $\mathrm{clif}_{(X,g)}(\alpha)|_{X\setminus B(Z,\eps)} = \mathrm{clif}_{(X,\tilde{g})}(\alpha)|_{X\setminus B(Z,\eps)}$ and an estimate $|\mathrm{clif}_{(X,g)}(\alpha)-\mathrm{clif}_{(X,\tilde{g})}(\alpha)| =  |\alpha|\cdot O(R_p^2)$ on $B(p,\eps)$ for any $p\in Z$,
				where $\mathrm{clif}_{(X,g)}$ and $\mathrm{clif}_{(X,\tilde{g})}$ denote the Clifford product with respect to $g$ and $\tilde{g}$ respectively.
			\end{Lem}
		
			We also take a direct sum of the flat Dirac monopoles $(V',h',A',\Phi')$ on $(B(Z,\eps) \setminus Z,g')$ to be $(V',h',A',\Phi')|_{B(p,\eps)} = \bigoplus_{i=1}^r (L_{k_{p,i}},h_{k_{p,i}},A_{k_{p,i}},\Phi_{k_{p,i}})$ for any $p\in Z$.
			By Corollary \ref{Cor:Appro. by Dirac monopole},
			there exists a unitary isomorphism $\varphi:V|_{B(Z,\eps) \setminus Z} \simeq V'$ such that the estimates in Corollary \ref{Cor:Appro. by Dirac monopole} are satisfied.
			We set a connection $\tilde{A}:=(1-\rho)A + \rho\cdot \varphi^{\ast}A'$ and an endomorphism $\tilde{\Phi}:=(1-\rho)\Phi + \rho\cdot \varphi^{\ast}\Phi'$.
			Then for each $p\in Z$ the restriction $(V,h,\tilde{A},\tilde{\Phi})|_{B(p,\eps/2)\setminus\{p\}}$ is a direct sum of the flat Dirac monopoles,
			and $|A-\tilde{A}|$ and $|\Phi- \tilde{\Phi}|$ are bounded on $X\setminus Z$.
			
			We denote by $\tilde{\dirac}^{\pm}_{(A,\Phi)}$ and $\tilde{\dirac}^{\pm}_{(\tilde{A},\tilde{\Phi})}$ the Dirac operators of the tuples $(V,h,A,\Phi)$ and $(V,h,\tilde{A},\tilde{\Phi})$ with respect to the metrics $\tilde{g}$ respectively.
			In Proposition \ref{Prp:formal adj. eq. ana. adj.},
			we show the Fredholmness of $\tilde{\dirac}^{\pm}_{(\tilde{A},\tilde{\Phi})}$.
			Consequently,
			we will prove the Fredholmness of $\dirac^{\pm}_{(A,\Phi)}$ in Theorem \ref{Thm:Index of modi. and orig. Dirac}.
			\begin{Prp}\label{Prp:Relich thm for modified monopole}
				The injection maps $\mathrm{Dom}(\tilde{\dirac}^{\pm}_{(\tilde{A},\tilde{\Phi})}) \to L^2(X,V\otimes S_X)$ are compact.
			\end{Prp}
			\begin{proof}
				The norm $||s||_{1} := ||s|_{X\setminus B(Z,\eps/8)}||_{L^2} + ||s|_{B(Z,\eps/4)^{\ast}}||_{L^2}$ on $L^2(X,V\otimes S_X)$ is equivalent to the ordinary $L^2$-norm on $X$.
				By the Rellich-Kondrachov theorem,
				the restriction maps $\mathrm{Dom}(\tilde{\dirac}^{\pm}_{(\tilde{A},\tilde{\Phi})})\ni s \to s|_{X\setminus B(Z,\eps/8)} \in L^2(X\setminus B(Z,\eps/8), S_X\otimes V)$ are compact.
				By Corollary \ref{Cor:Relich type thm for Dirac monopole},
				the restriction maps $\mathrm{Dom}(\tilde{\dirac}^{\pm}_{(\tilde{A},\tilde{\Phi})})\ni s \to s|_{B(Z,\eps/4)^{\ast}} \in L^2(B(Z,\eps/4), S_X\otimes V)$ are also compact.
				Hence the injection maps $\mathrm{Dom}(\tilde{\dirac}^{\pm}_{(\tilde{A},\tilde{\Phi})}) \to L^2(X,V\otimes S_X)$ are compact.
			\end{proof}
			\begin{Prp}\label{Prp:formal adj. eq. ana. adj.}
				The Dirac operators $\tilde{\dirac}^{\pm}_{(\tilde{A},\tilde{\Phi})}: L^2(X\setminus Z,V\otimes S_X)\to L^2(X \setminus Z,V\otimes S_X)$ are closed Fredholm operators and adjoint to each other.
			\end{Prp}
			\begin{proof}
				We show that $\tilde{\dirac}^{\pm}_{(\tilde{A},\tilde{\Phi})}$ are adjoint to each other.
				For a densely defined closed operator $F$,
				we denote by $F^{\ast}$ the adjoint of $F$.
				Take $\alpha\in\mathrm{Dom}\left((\tilde{\dirac}^{\pm}_{(\tilde{A},\tilde{\Phi})})^{\ast}\right)$.
				Then we have $\i<(\tilde{\dirac}^{\pm}_{(\tilde{A},\tilde{\Phi})})^{\ast}(\alpha),\varphi>_{L^2} = \i<\alpha, \tilde{\dirac}^{\pm}_{(\tilde{A},\tilde{\Phi})}(\varphi)>_{L^2}$ for any $\varphi\in C^{\infty}_{0}(X\setminus Z,V\otimes S_X)$,
				where $C^{\infty}_{0}(X\setminus Z,V\otimes S_X)$ denotes the set of compact-supported smooth sections.
				Therefore $\alpha \in \mathrm{Dom}(\tilde{\dirac}^{\mp}_{(\tilde{A},\tilde{\Phi})})$ and $(\tilde{\dirac}^{\pm}_{(\tilde{A},\tilde{\Phi})})^{\ast}(\alpha) = \tilde{\dirac}^{\mp}_{(\tilde{A},\tilde{\Phi})}(\alpha)$.
				We show the converse.
				Take $a\in\mathrm{Dom}(\tilde{\dirac}^{\mp}_{(\tilde{A},\tilde{\Phi})})$ and $b\in\mathrm{Dom}(\tilde{\dirac}^{\pm}_{(\tilde{A},\tilde{\Phi})})$.
				Because of the elliptic regularity, Proposition \ref{Prp:Ker of Dirac monopole} and Proposition \ref{Prp:Dol. lemma for Dirac monopoles},
				we obtain $|a|,|b| \in L^3(X\setminus Z)$.
				Let $\kappa:\mathbb{R}\to [0,1]$ be a smooth function that satisfies the conditions $\kappa((-\infty,-1])=\{0\}$, $\kappa([-1/2,\infty))=\{1\}$.
				Set $\psi_n(x) = \kappa(n\cdot\log(d_{\tilde{g}}(x,Z)))$ for $n\in\mathbb{N}$,
				where we set $d_{\tilde{g}}(x,Z) := \mathrm{min}\{d_{\tilde{g}}(x,p) \mid p\in Z\}$.
				Since $\psi_n a$ has a compact support on $X\setminus Z$,
				we have $\i<\psi_n a, \tilde{\dirac}^{\pm}_{(\tilde{A},\tilde{\Phi})}(b)>_{L^2} = \i<\tilde{\dirac}^{\mp}_{(\tilde{A},\tilde{\Phi})}(\psi_n a), b>_{L^2} = \i<\psi_n \tilde{\dirac}^{\mp}_{(\tilde{A},\tilde{\Phi})}(a), b>_{L^2} + \i<\mathrm{clif}_{X}(d\psi_n)a,b>_{L^2}$.
				Since we have $|(\kappa(nx))'| \leq (x|\log(x)|)^{-1}\cdot ||\kappa'||_{L^\infty}$ for $0<x<1$,
				$|d\psi_n|$ is dominated by an $L^3$-function that is independent of $n$.
				Hence we obtain $\i<a,\tilde{\dirac}^{\pm}_{(\tilde{A},\tilde{\Phi})}(b)>_{L^2} = \i<\tilde{\dirac}^{\mp}_{(\tilde{A},\tilde{\Phi})}(a),b>_{L^2}$ by the dominated convergence theorem.
				Therefore $a\in \mathrm{Dom}\left((\tilde{\dirac}^{\pm}_{(\tilde{A},\tilde{\Phi})})^{\ast}\right)$ and $(\tilde{\dirac}^{\pm}_{(\tilde{A},\tilde{\Phi})})^{\ast}(a)=\tilde{\dirac}^{\mp}_{(\tilde{A},\tilde{\Phi})}(a)$.
				
				We show that the kernel of $\tilde{\dirac}^{\pm}_{(\tilde{A},\tilde{\Phi})}$ is finite-dimensional.
				By Proposition \ref{Prp:Relich thm for modified monopole},
				the identity map of $\mathrm{Ker}(\tilde{\dirac}^{\pm}_{(\tilde{A},\tilde{\Phi})})$ is a compact operator.
				Hence we obtain $\mathrm{dim}(\mathrm{Ker}(\tilde{\dirac}^{\pm}_{(\tilde{A},\tilde{\Phi})}))<\infty$.
				Since the Dirac operators $\tilde{\dirac}^{\pm}_{(\tilde{A},\tilde{\Phi})}$ are adjoint to each other,
				the claim $\mathrm{dim}(R(\tilde{\dirac}^{\pm}_{(\tilde{A},\tilde{\Phi})})^{\perp})<\infty$ can be deduced from $\mathrm{dim}(\mathrm{Ker}(\tilde{\dirac}^{\pm}_{(\tilde{A},\tilde{\Phi})}))<\infty$,
				where $R(\cdot)$ means the range of the operator and $\perp$ means the orthogonal complement in $L^2$.
				
				To prove that $R(\tilde{\dirac}^{\pm}_{(\tilde{A},\tilde{\Phi})})$ is closed,
				it suffices to show that there exists a constant $C>0$ such that the condition $||s||_{L^2} < C||\tilde{\dirac}^{\pm}_{(\tilde{A},\tilde{\Phi})}(s)||_{L^2}$ holds for any $s \in \mathrm{Dom}(\tilde{\dirac}^{\pm}_{(\tilde{A},\tilde{\Phi})}) \cap \left(\mathrm{Ker}(\tilde{\dirac}^{\pm}_{(\tilde{A},\tilde{\Phi})})\right)^{\perp}$.
				Suppose that there is no such a constant $C>0$,
				then we can take a sequence $\{s_n\}\subset \mathrm{Dom}(\tilde{\dirac}^{\pm}_{(\tilde{A},\tilde{\Phi})}) \cap \left(\mathrm{Ker}(\tilde{\dirac}^{\pm}_{(\tilde{A},\tilde{\Phi})})\right)^{\perp}$ such that the conditions $||s_n||=1$ and $||\tilde{\dirac}^{\pm}_{(\tilde{A},\tilde{\Phi})}(s_n)||_{L^2}<1/n$ are satisfied.
				By Proposition \ref{Prp:Relich thm for modified monopole},
				we may assume that $\{s_n\}$ converges to some $s_{\infty}\in L^2$.
				Since $||s_n||_{L^2}=1$ for any $n \in\mathbb{N}$,
				we have $s_{\infty} \in \left(\mathrm{Ker}(\tilde{\dirac}^{\pm}_{(\tilde{A},\tilde{\Phi})})\right)^{\perp} \setminus \{0\}$.
				However,
				we also have $\tilde{\dirac}^{\pm}_{(\tilde{A},\tilde{\Phi})}(s_n)\to 0$,
				and hence $\tilde{\dirac}^{\pm}_{(\tilde{A},\tilde{\Phi})}(s_{\infty})=0$,
				which is impossible.
				Therefore the condition holds for some $C>0$ and $R(\tilde{\dirac}^{\pm}_{(\tilde{A},\tilde{\Phi})})$ is closed.
			\end{proof}
			
			\begin{Thm}\label{Thm:Index of modi. and orig. Dirac}
				The Dirac operators $\dirac^{\pm}_{(A,\Phi)}$ are closed Fredholm operators and adjoint to each other,
				and they have the same indices of  $\tilde{\dirac}^{\pm}_{(\tilde{A},\tilde{\Phi})}$.
			\end{Thm}
			\begin{proof}
				Since $|A - \tilde{A}|$ and $|\Phi-\tilde{\Phi}|$ are bounded on $X\setminus Z$,
				by Proposition \ref{Prp:Relich thm for modified monopole}
				the operators $\tilde{\dirac}^{\pm}_{(A,\Phi)}$ are closed Fredholm operators and adjoint to each other,
				and they have the same indices of  $\tilde{\dirac}^{\pm}_{(\tilde{A},\tilde{\Phi})}$.
				
				We will prove that $\tilde{\dirac}^{\pm}_{(A,\Phi)}$ and $\dirac^{\pm}_{(A,\Phi)}$ have the same domains and indices.
				First we show $\mathrm{Dom}(\dirac^{\pm}_{(A,\Phi)}) \subset \mathrm{Dom}(\tilde{\dirac}^{\pm}_{(A,\Phi)})$.
				By Lemma \ref{Lem:Pertubation of Clif prod.},
				there exists $C_0>0$ such that the estimate
				\begin{equation}\label{Eqn:Est. of cov. deriv.}
					|\tilde{\dirac}^{\pm}_{(A,\Phi)}(s) - \dirac^{\pm}_{(A,\Phi)}(s)|
					<  C_0\,R_p^2\cdot|\nabla_{A_X\otimes A}(s)|
				\end{equation}
				holds on a neighborhood of $p\in Z$ for any $s\in\Gamma(X\setminus Z, S_X\otimes V)$.
				Let $\kappa:[0,\infty]:\to [0,1]$ be a smooth bump function satisfying
				\[
					\kappa(x)=
						\left\{
							\begin{array}{ll}
								0 & (1 \leq x)\\
								1 & (3/8 \leq x \leq 3/4)\\
								0 & (x \leq 1/3).
							\end{array}
						\right.
				\]
				For $\delta>0$,
				we set a function $\varphi_{\delta}:X \to [0,1]$ to be $\varphi_{\delta}(x) = \kappa(\delta^{-1}d_g(x,Z))$.
				By the Weitzenb\"{o}ck formula $\nabla^{\bigstar_g}_{A_X\otimes A}\nabla_{A_X\otimes A} = \dirac^{-}_{(A,\Phi)}\dirac^{+}_{(A,\Phi)} - \Phi^2 + Sc(g)$,
				we have
				$||\nabla_{A_{S_X}\otimes A}(\varphi_{\delta}s)||_{L^2}^2 = ||\dirac^{+}_{(A,\Phi)}(\varphi_{\delta}s)||_{L^2}^2 + ||\Phi(\varphi_{\delta}s)||_{L^2}^2 + \int_X Sc(g)|\varphi_{\delta}s|^2 d\mathrm{vol}_M$ for any $s\in\Gamma(X\setminus Z, V\otimes S_X)$,
				where $\nabla^{\bigstar_g}_{A_X\otimes A}$ is the formal adjoint of $\nabla_{A_X\otimes A}$ with respect to $g$ and $Sc(g)$ is the scalar curvature of $g$.
				Therefore,
				there exists $C_1>0$ such that for any sufficiently small $\delta>0$ we have $||\nabla_{A_{S_X}\otimes\,A}(s)||_{L^2(U_p(3\delta/8,3\delta/4))} \leq C_1( ||\dirac^{\pm}_{(A,\Phi)}(s)||_{L^2(U_p(\delta/3,\delta))} + \delta^{-1}||s||_{L^2(U_p(\delta/3,\delta))})$ holds for any $s\in\Gamma(X\setminus Z, V\otimes S_X)$,
				where we put $U_p(\delta_1,\delta_2):= \{x\in X \mid \delta_1 < d_g(p,x) < \delta_2\}$.
				We set $\delta_i:= 4\eps/(3\cdot2^i)$ for $i\in\mathbb{Z}_{\geq0}$.
				Then we have
				\begin{align*}
					&||\tilde{\dirac}^{\pm}_{(A,\Phi)}(s) - \dirac^{\pm}_{(A,\Phi)}(s)||_{L^2(B(p,\eps))}\\
					\leq\,& C_0 ||R_p^2\cdot\nabla_{A_X\otimes A}(s)||_{L^2(B(p,\eps))}\\
					\leq\,& C_0\sum_{i=0}^{\infty}||R_p^2\cdot\nabla_{A_X\otimes A}(s)||_{L^2(U_p(3\delta_i/8, 3\delta_i/4))}\\
					\leq\,& C_0C_1\sum_{i=0}^{\infty}\left\{\delta_i^2 ||\dirac^{\pm}_{(A,\Phi)}(s)||_{L^2(U_p(\delta_i/3,\delta_i))} + \delta_i||s||_{L^2(U_p(\delta_i/3,\delta_i))}\right\}\\
					\leq\,& C_0C_1\Bigl(\delta_0^2||\dirac^{\pm}_{(A,\Phi)}(s)||_{L^2(B(p,\delta_0))} + \delta_0||s||_{L^2(B(p,\delta_0))}\Bigr)
				\end{align*}
				Hence there exists $C_2=C_2(\eps) >0$ such that $||\tilde{\dirac}^{\pm}_{(A,\Phi)}(s) - \dirac^{\pm}_{(A,\Phi)}(s)||_{L^2} \leq C_2(||s||_{L^2} + ||\dirac^{\pm}_{(A,\Phi)}(s)||_{L^2})$,
				and we have $C_2 = O(\eps)\;(\eps\to0)$.
				Hence we obtain $\mathrm{Dom}(\dirac^{\pm}_{(A,\Phi)}) \subset \mathrm{Dom}(\tilde{\dirac}^{\pm}_{(A,\Phi)})$.
				We show the converse.
				Let $\tilde{A}_X$ denote the connection on $S_X$ induced by the Levi-Civita connection of $(X,\tilde{g})$.
				By the definition of Dirac-type singularity,
				we have $|\nabla_A(\Phi)| = |F(A)| = O(R_p^{-2})$ around $p\in Z$.
				Therefore from the Weitzenb\"{o}ck formula $\nabla^{\bigstar_{\tilde{g}}}_{\tilde{A}_X\otimes A}\nabla_{\tilde{A}_X\otimes A} = \dirac^{-}_{(A,\Phi)}\dirac^{+}_{(A,\Phi)} - \Phi^2 + \mathrm{clif}(\nabla_{A}(\Phi) - \ast_{\tilde{g}}F(A))$ and a similar argument as above,
				it follows that there exists $C_3 = C_3(\eps)>0$ such that $||\tilde{\dirac}^{\pm}_{(A,\Phi)}(s) - \dirac^{\pm}_{(A,\Phi)}(s)||_{L^2} \leq C_3(||s||_{L^2} + ||\tilde{\dirac}^{\pm}_{(A,\Phi)}(s)||_{L^2})$,
				and $C_3 = O(\eps)\;(\eps\to0)$.
				Therefore $\mathrm{Dom}(\dirac^{\pm}_{(A,\Phi)}) = \mathrm{Dom}(\tilde{\dirac}^{\pm}_{(A,\Phi)})$.
				Moreover,
				Their graph norms are also equivalent.
				It is a well-known fact that sufficiently small deformations of Fredholm operators remain Fredholm.
				Hence the operators $\dirac^{\pm}_{(A,\Phi)}$ are closed Fredholm operators,
				and they have the same indices as ones of $\tilde{\dirac}^{\pm}_{(\tilde{A},\tilde{\Phi})}$.
			\end{proof}
		\subsection{An index calculation on a compact $3$-folds}
			\subsubsection{A lift of singular monopoles to closed $4$-folds}
				For an arbitrary $3$-fold $N$ and a principal $S^1$-bundle $P$ defined on outside of a point $x\in N$,
				we set $\mathrm{deg}_{x}(P) := \int_{\partial B}c_1(P)$,
				where $B$ is a small neighborhood of $x$. 
				 
				We take a finite subset $Z'\subset X$ satisfying the conditions $|Z'|=|Z|$ and $Z \cap Z' = \emptyset$,
				and set $\tilde{Z}=Z\cup Z'$.
				By the Mayer-Vietoris exact sequence induced by the open covering $X = B_{\eps}(\tilde{Z})\cup (X \setminus \tilde{Z})$,
				we can prove that there exists a principal $S^1$-bundle $\pi:P\to X \setminus \tilde{Z}$ such that we have $\mathrm{deg}_{p}(P)=-1$ for $p\in Z$ and $\mathrm{deg}_{p'}(P)=1$ for $p'\in Z'$.
				We take a metric $\hat{g}$ on $X$ that is flat on $B(\tilde{Z},\eps/2)$.
				Let $f:X\setminus\tilde{Z}\to \mathbb{R}_+$ be a smooth function.
				Let $\omega\in\Omega^1(P,\mathbb{R})$ be a connection of $P$.
				We assume that for any $p\in Z$\,(resp. $Z'$) the tuple $((P,\omega)\times_{S^1}(\mathbb{C},h_{\mathbb{C}}),-\sqrt{-1}f)|_{B(p,\eps/2)}$\;(resp. $((P,\omega)\times_{S^1}(\mathbb{C},h_{\mathbb{C}}),\sqrt{-1}f)|_{B(p,\eps/2)}$) is the flat Dirac monopole of weight $-1$\,(resp. $1$) with respect to $\hat{g}$.
				Set a one-form $\xi:= \omega/\pi^{\ast}f$ and a metric $g_P := \pi^{\ast}\hat{g} + \xi^2$ on P.
				We choose the global $4$-form $-\xi\wedge \pi^{\ast}\mathrm{vol}_{(X,\hat{g})}$ as the orientation of $P$.
				\begin{Prp}\label{Prp:lift of ori. and spin.}
					The following claims are satisfied.
					\begin{itemize}
						\item 
							The $4$-fold $P$ has the spin structure induced by the one of $X$.
						\item
							Let $v$ be a vector field on $X$.
							By the isomorphism $TP=\mathbb{R}\partial_{\theta}\oplus\pi^{\ast}TX$ induced by $\omega$,
							we regard $\pi^{\ast}v$ as a vector field on $P$.
							Then for $F\in C^{\infty}(X)$ we have $\pi^{\ast}(v\cdot F)= \pi^{\ast}v\cdot \pi^{\ast}F$.
						\item
							For the spinor bundles $S^{\pm}$,
							we have the unitary isomorphism $S^{\pm}_P \simeq \pi^{\ast}(S_X)$.
						\item
							Under the above isomorphisms,
							the Clifford product on $P$ can be represented as follows:
							\begin{alignat*}{3}
								&\mathrm{clif}_{P}(\xi)&=&
									\left(
										\begin{array}{cc}
											0 & \mathrm{Id}\\
											-\mathrm{Id} & 0
										\end{array}
									\right)&&\\
								&\mathrm{clif}_{P}(\pi^{\ast}\alpha)\ &=&
									\left(
									\begin{array}{cc}
									0 & \mathrm{clif}_{X}(\alpha)\\
									\mathrm{clif}_{X}(\alpha) & 0
									\end{array}
									\right)&&\ \ \ \ \ (\alpha\in\Gamma(X,\Omega^1(X))).
							\end{alignat*}
					\end{itemize}
				\end{Prp}
				\begin{proof}
					Let $i:SO(3)\to SO(4)$ be the injection induced by $\mathbb{R}^3 \ni x \to (0,x)\in\mathbb{R}^4$,
					and take the homomorphism $\rho:Spin(3)\to Spin(4)$ to be the lift of $i$.
					Set $\mathscr{S}_{P}:= \pi^{\ast}\mathscr{S} \times_{\rho} Spin(4)$.
					Then we have $\mathscr{S}_{P}\times_{Spin(4)} (\mathbb{R}^4,g_{4,\mathrm{Euc}}) \simeq \left(P\times(\mathbb{R},g_{1,\mathrm{Euc}})\right) \oplus (\pi^{\ast}TX,\pi^{\ast}g) \simeq TP$.
					Hence $\mathscr{S}_P$ is a spin structure on $P$.
					The second claim is trivial from some direct calculations.
					
					We have the isomorphisms $Spin(3)\simeq SU(2)$ and $Spin(4)\simeq SU(2)_{+}\times SU(2)_{-}$.
					Under this isomorphism,
					we have $\rho(g)=(g,g)$.
					Hence we obtain the unitary isomorphism $S^{\pm}_P \simeq \pi^{\ast}(S_X)$.
					The last claim easily follows from the third one. 
				\end{proof}
				We take another metric $\tilde{g}_P:= \pi^{\ast}f\cdot g_P$.
				For $p\in Z$,
				the restriction $\pi:\pi^{\ast}(B(p,\eps/2)\setminus\{p\}) \to B(p,\eps/2)\setminus\{p\}$ can be identified with the Hopf fibration $(\mathbb{R}^4\setminus\{0\})\to (\mathbb{R}^3\setminus\{0\})$.
				For $p'\in Z'$,
				we can also identify $\pi:\pi^{\ast}(B(p',\eps/2)\setminus\{p'\}) \to B(p',\eps/2)\setminus\{p'\}$ with the inverse-oriented Hopf fibration $(-\mathbb{R}^4\setminus\{0\})\to (\mathbb{R}^3\setminus\{0\})$,
				where $-\mathbb{R}^4$ is the differentiable manifold $\mathbb{R}^4$ with the inverse orientation of the standard one of $\mathbb{R}^4$.
				Hence by taking the one-point compactification on the closure of each $\pi^{\ast}(B(p,\eps/2)\setminus \{p\})$,
				we obtain a closed $4$-fold $\tilde{P}$ equipped with an $S^1$-action.
				Then $\tilde{g}_P$ can be prolonged to a metric on $\tilde{P}$ as in Lemma \ref{Lem:prepare of Pauly condition}.
				We extend the projection $\pi:P\to X\setminus \tilde{Z}$ to the smooth map $\tilde{P} \to X$,
				and we denote this map by the same letter $\pi$ by abuse of notation.
				Set $Z_4:=\pi^{-1}(Z)$, $Z'_4:=\pi^{-1}(Z')$ and $\tilde{Z}_4:=\pi^{-1}(\tilde{Z})$.
				Then $\pi|_{\tilde{Z}_4}:\tilde{Z}_4 \to Z_4$ is a bijection.
				We have $\tilde{P}=P \sqcup \tilde{Z}_4$ and $\mathrm{codim}(\tilde{P},\tilde{Z})=4$.
				Hence we obtain isomorphisms $\pi_1(P)\simeq \pi_1(\tilde{P})$ and $H^2(P,\mathbb{Z}/2\mathbb{Z})\simeq H^2(\tilde{P},\mathbb{Z}/2\mathbb{Z})$.
				Therefore the orientation and the spin structure of $P$ induce the unique ones of $\tilde{P}$.
				Hence we obtain the following lemma. 
				\begin{Lem}\label{Lem:lift of clif. Prod.}
					We have the unitary isomorphisms $S^{\pm}_{\tilde{P}}|_P \simeq (\pi^{\ast}S_X)|_P$.
					Under these isomorphisms,
					we have $\mathrm{clif}_{\tilde{P}}(v)|_{P} = \pi^{\ast}f^{-1/2}\cdot\mathrm{clif}_P(v)$ for $v\in \Omega^{1}(\tilde{P})$.
				\end{Lem}
				
				For the Dirac-type singular monopole $(V,h,A,\Phi)$ on $(X,Z)$,
				we take a connection $\hat{A}$ and a skew-Hermitian endomorphism $\hat{\Phi}$ that satisfy the following conditions.
				\begin{itemize}
					\item 
						For any $p\in Z$,
						$(V,h,\hat{A},\hat{\Phi})|_{B(p,\eps/2)\setminus\{p\}}$ is a direct sum of the flat Dirac monopoles with respect to the metric $\hat{g}$.
					\item 
						For any $p'\in Z'$,
						$(V,h,\hat{A})|_{B(p',\eps/2)}$ is a flat unitary bundle and $\hat{\Phi}|_{B(p',\eps/2)}=0$.
					\item 
						The differences $|A-\hat{A}|$, $|\Phi - \hat{\Phi}|$ are bounded on $X\setminus \tilde{Z}$.
				\end{itemize}
				We denote by $\hat{\dirac}^{\pm}_{(\hat{A},\hat{\Phi})}$ the Dirac operators of $(V,h,\hat{A},\hat{\Phi})$ with respect to the metric $\hat{g}$.
				By the same argument as Proposition \ref{Prp:formal adj. eq. ana. adj.} and Theorem \ref{Thm:Index of modi. and orig. Dirac},
				$\hat{\dirac}^{\pm}_{(\hat{A},\hat{\Phi})}$ are Fredholm and adjoint to each other,
				and the indices of $\hat{\dirac}^{\pm}_{(\hat{A},\hat{\Phi})}$ are the same as the ones of $\dirac^{\pm}_{(A,\Phi)}$.
				
				We set $(V_4,h_4,A_4):= (\pi^{\ast}V,\pi^{\ast}h,\pi^{\ast}\hat{A} - \xi\otimes \pi^{\ast}\hat{\Phi})$ on $P\sqcup Z'_4$.
				By Proposition \ref{Prp:Pauly condition},
				$(V_4,h_4,A_4)$ can be prolonged over $\tilde{P}$,
				and we denote it by the same symbols.
				Let $\dirac^{\pm}_{A_4}:\Gamma(\tilde{P},S^{\pm}_{\tilde{P}}\otimes V_4)\to \Gamma(\tilde{P},S^{\mp}_{\tilde{P}}\otimes V_4)$ be the Dirac operators of $(V_4,h_4,A_4)$.
				For a section $s\in \Gamma(X\setminus Z,V\otimes S_X)$,
				we have $||\pi^{\ast}s||^2_{L^2(\tilde{P},\tilde{g}_P)}= 2\pi||\sqrt{f}s||^2_{L^2(X,\hat{g})}$.
				Hence the operator $\pi^{\dagger}(s):= \pi^{\ast}(\sqrt{2\pi f^{-1}}\,s)$ preserves the $L^2$-norms.
				Since $P$ is a principal $S^1$-bundle on $X$,
				$\pi^{\dagger}$ is an isometric isomorphism from $L^2(X\setminus \tilde{Z},V\otimes S_X)$ to $L^2(\tilde{P},V_4 \otimes S^{\pm}_{\tilde{P}})^{S^1}$,
				where $L^2(\tilde{P},V_4 \otimes S_{\tilde{P}})^{S^1}$ is the closed subspace of $L^2(\tilde{P},V_4 \otimes S_{\tilde{P}})$ consisting of $S^1$-invariant sections.
				For $i=1,2$,
				take smooth functions $\lambda^{\pm}_{i}:X\setminus \tilde{Z} \to \mathbb{R}_+$ satisfying the following conditions.
				\begin{itemize}
					\item 
						The equality $\lambda^{\pm}_{1}\lambda^{\pm}_{2} = f^{-1/2}$ holds.
					\item
						The equality  $\lambda^{\pm}_{1} = \lambda^{\mp}_{2}$ holds.
					\item
						For any $p\in Z$,
						$\lambda^{+}_1|_{B(p,\eps)\setminus\{p\}} = 1$.
					\item
						For any $p'\in Z'$,
						$\lambda^{+}_2|_{B(p',\eps)\setminus\{p'\}} = 1$.
				\end{itemize}
				By Lemma \ref{Lem:lift of clif. Prod.} and Proposition \ref{Prp:Local lift of Dirac op.},
				there exist compact-supported smooth endomorphisms $\epsilon^{\pm}\in\Gamma(X\setminus \tilde{Z}, \mathrm{End}(S_X\otimes V))$ such that we have $\pi^{\dagger}(\epsilon^{\pm})(s) = \dirac^{\pm}_{A_4} \comp \pi^{\dagger}(s) - \pi^{\dagger}\comp (\lambda^{\pm}_1\hat{\dirac}^{\pm}_{(\hat{A},\hat{\Phi})}\lambda^{\pm}_2)(s)$ for any $s\in\Gamma(X\setminus Z, S_X\otimes V)$.
				Let $D^{\pm}$ be the differential operator $\lambda^{\pm}_1\hat{\dirac}^{\pm}_{(\hat{A},\hat{\Phi})}\lambda^{\pm}_2 + \epsilon^{\pm}$ on $X\setminus\tilde{Z}$.
				We denote by $\mathrm{Ind}(\dirac^{\pm}_{A_4})^{S^1}$ the $S^1$-equivariant index of the closed operator $\dirac^{\pm}_{A_4}:L^2(\tilde{P},V_4 \otimes S^{\pm}_{\tilde{P}})^{S^1} \to L^2(\tilde{P},V_4 \otimes S^{\mp}_{\tilde{P}})^{S^1}$.
				\begin{Prp}\label{Prp:L^2-lift of Dirac op.}
					Under the isometric isomorphism $\pi^{\dagger}$,
					the operators $D^{\pm}$ and $\dirac^{\pm}_{A_4}$ determine the same closed operators respectively.
					In particular,
					the operators $D^{\pm}$ are closed Fredholm operator adjoint to each other,
					and satisfy $\mathrm{Ind}(D^{\pm}) =\mathrm{Ind}(\dirac^{\pm}_{A_4})^{S^1}$.
				\end{Prp}
				\begin{proof}
					We take an arbitrary $a\in \mathrm{Dom}(\dirac^{\pm}_{A_4})^{S^1}$, and set $b := \dirac^{\pm}_{A_4}(a)$.
					We will show $(\pi^{\dagger})^{-1}(a) \in \mathrm{Dom}(D^{\pm})$ and $D^{\pm}((\pi^{\dagger})^{-1}(a)) = (\pi^{\dagger})^{-1}(b)$.
					Let $\varphi$ be a compact-supported smooth section of $V\otimes S_X$ on $X\setminus \tilde{Z}$.
					Then $\pi^{\dagger}(\varphi)$ also has a compact support.
					Hence we have $\i<a,(\dirac^{\pm}_{A_4})^{\bigstar}(\pi^{\dagger}(\varphi))>_{L^2}=\i<b,\pi^{\dagger}(\varphi)>_{L^2}$.
					Since $(\pi^{\dagger})^{-1}$ is isometric,
					we obtain $\i<(\pi^{\dagger})^{-1}(a),(D^{\pm})^{\bigstar}(\varphi)> = \i<(\pi^{\dagger})^{-1}(b),\varphi>$.
					Therefore we have $(\pi^{\dagger})^{-1}(a) \in \mathrm{Dom}(D^{\pm})$ and $D^{\pm}((\pi^{\dagger})^{-1}(a)) = (\pi^{\dagger})^{-1}(b)$.
					We prove the converse.
					We take an arbitrary $c\in \mathrm{Dom}(D^{\pm})$, and set $d := D^{\pm}(c)$.
					Let $\chi$ be a compact-supported smooth section of $V_4\otimes S_{\tilde{P}}$ on $\tilde{P}\setminus \tilde{Z}_4$.
					We take the orthogonal decomposition $\chi=\chi^{S^1} + \chi^{\bot}\in L^2(\tilde{P},V_4 \otimes S^{\mp}_{\tilde{P}})^{S^1} \oplus (L^2(\tilde{P},V_4 \otimes S^{\mp}_{\tilde{P}})^{S^1})^{\bot}$.
					Then $\chi^{S^1}$ and $\chi^{\bot}$ are also compact-supported smooth sections on $\tilde{P}\setminus \tilde{Z}_4$,
					and we have $(\dirac^{\pm}_{A_4})^{\bigstar}(\chi^{S^1}) \in L^2(\tilde{P},V_4 \otimes S^{\mp}_{\tilde{P}})^{S^1}$ and $(\dirac^{\pm}_{A_4})^{\bigstar}(\chi^{\bot}) \in (L^2(\tilde{P},V_4 \otimes S^{\mp}_{\tilde{P}})^{S^1})^{\bot}$.
					Hence we obtain $\i<\pi^{\dagger}(c),(\dirac^{\pm}_{A_4})^{\bigstar}(\chi)>_{L^2} = \i<\pi^{\dagger}(c),(\dirac^{\pm}_{A_4})^{\bigstar}(\chi^{S^1})>_{L^2} = \i<c,(\pi^{\dagger})^{-1}((\dirac^{\pm}_{A_4})^{\bigstar}(\chi^{S^1}))>_{L^2} = \i<c, (D^{\pm})^{\bigstar}((\pi^{\dagger})^{-1}(\chi^{S^1}))>_{L^2} = \i<d, (\pi^{\dagger})^{-1}(\chi^{S^1})>_{L^2} = \i<\pi^{\dagger}(d),\chi^{S^1}>_{L^2} = \i<\pi^{\dagger}(d),\chi>_{L^2}$.
					Therefore $\dirac^{\pm}_{A_4}(\pi^{\dagger}(c)) = \pi^{\dagger}(d)$ holds on $P=\tilde{P}\setminus \tilde{Z}_4$.
					Here we prepare the following lemma.
					\begin{Lem}\label{Lem:prolong of sol. of Dirac eq.}
						Take arbitrary $u\in L^2(\tilde{P},V_4\otimes S^{\pm}_{\tilde{P}})$ and $v\in L^2(\tilde{P},V_4\otimes S^{\mp}_{\tilde{P}})$.
						If $u$ and $v$ satisfy $\dirac^{\pm}_{A_4}(u)=v$ on $P$,
						then we have $\dirac^{\pm}_{A_4}(u)=v$ on whole $\tilde{P}$.
					\end{Lem}
					If we admit this lemma,
					then we obtain $\dirac^{\pm}_{A_4}(\pi^{\dagger}(c)) = \pi^{\dagger}(d)$ on $\tilde{P}$.
					Hence the proof is complete. 
				\end{proof}
				\begin{proof}[{\upshape \textbf{proof of Lemma \ref{Lem:prolong of sol. of Dirac eq.}}}]
					Take $\varphi\in \Gamma(\tilde{P},V_4\otimes S^{\mp}_{\tilde{P}})$.
					Let $\kappa:\mathbb{R}\to [0,1]$ be a smooth function that satisfies $\kappa((-\infty,-1))=\{0\}$ and $\kappa((-1/2,\infty))=\{1\}$.
					Set $\psi_n:\tilde{P}\to[0,1]$ to be $\psi_n(x):= \kappa(n\log(d_{\tilde{g}_P}(x,\tilde{Z}_4)))$ for $n\in\mathbb{N}$.
					Then $\psi_n\cdot\varphi$ has a compact support on $\tilde{P}\setminus\tilde{Z}_4$.
					Hence we obtain $\i<u,(\dirac^{\pm}_{A_4})^{\bigstar}(\psi_n\cdot\varphi)>_{L^2}= \i<u,\psi_n\cdot (\dirac^{\pm}_{A_4})^{\bigstar}(\varphi)>_{L^2} + \i<u,\mathrm{clif}_{\tilde{P}}(d\psi_n)\varphi>_{L^2} = \i<v,\psi_n\cdot\varphi>_{L^2}$.
					Since we have an estimate $|\kappa'(nx)|\leq (x|\log(x)|)^{-1}||\kappa'||_{L^\infty}$ for $0<x<1$,
					$|d\psi_n|$ is dominated by an $L^2$-function that is independent of $n$.
					Therefore we obtain $\i<u,(\dirac^{\pm}_{A_4})^{\bigstar}(\varphi)>_{L^2}= \i<v,\varphi>$ by the dominated convergence theorem.
				\end{proof}
			We will associate the $S^1$-invariant indices of $\dirac^{\pm}_{A_4}$ and the indices of $\hat{\dirac}^{\pm}_{(\hat{A},\hat{\Phi})}$.
			\begin{Prp}
				We have $\mathrm{Ind}(\dirac^{\pm}_{A_4})^{S^1} = \mathrm{Ind}(\hat{\dirac}^{\pm}_{(\hat{A},\hat{\Phi})})$.
			\end{Prp}
			\begin{proof}
				If we prove $\mathrm{Ind}(\dirac^{+}_{A_4})^{S^1} = \mathrm{Ind}(\hat{\dirac}^{+}_{(\hat{A},\hat{\Phi})})$,
				then we obtain $\mathrm{Ind}(\dirac^{-}_{A_4})^{S^1} = -\mathrm{Ind}(\dirac^{+}_{A_4})^{S^1} = -\mathrm{Ind}(\hat{\dirac}^{+}_{(\hat{A},\hat{\Phi})}) =\mathrm{Ind}(\hat{\dirac}^{-}_{(\hat{A},\hat{\Phi})})$
				because $\hat{\dirac}^{\pm}_{(\hat{A},\hat{\Phi})}$ are adjoint to each other.
				Hence we only need to prove $\mathrm{Ind}(\dirac^{+}_{A_4})^{S^1} = \mathrm{Ind}(\hat{\dirac}^{+}_{(\hat{A},\hat{\Phi})})$.
				By Proposition \ref{Prp:L^2-lift of Dirac op.},
				it suffices to show $\mathrm{Ind}(\hat{\dirac}^{+}_{(\hat{A},\hat{\Phi})}) = \mathrm{Ind}(D^{+})$.
				Since the support of $\epsilon^{+}$ is compact in $X\setminus Z$,
				$\lambda^{+}_1\hat{\dirac}^{+}_{(\hat{A},\hat{\Phi})}\lambda^{+}_2$ is a closed Fredholm operator and it has the same index as $D^{+}$.
				By the same asymptotic analysis in Proposition \ref{Prp:Ker of Dirac monopole},
				for any solutions $s\in \Gamma(X\setminus \tilde{Z}, S_X\otimes V)$ of the equation $\hat{\dirac}^{+}_{(\hat{A},\hat{\Phi})}(s)=0$,
				we have $s\in L^2$ if and only if $(\lambda^{+}_2)^{-1}s\in L^2$. 
				Hence we have the natural equality $\mathrm{Ker}(\lambda^{+}_1\hat{\dirac}^{+}_{(\hat{A},\hat{\Phi})}\lambda^{+}_2)\cap L^2 = (\lambda^{+}_2)^{-1}\cdot(\mathrm{Ker}(\hat{\dirac}^{+}_{(\hat{A},\hat{\Phi})})\cap L^2)$,
				where $(\lambda^{+}_2)^{-1}\cdot(\mathrm{Ker}(\hat{\dirac}^{+}_{(\hat{A},\hat{\Phi})})\cap L^2)$ means the set $\{(\lambda^{+}_2)^{-1}\cdot s \mid s\in\mathrm{Ker}(\hat{\dirac}^{+}_{(\hat{A},\hat{\Phi})})\cap L^2\}$.
				By a similar way,
				we also have $\mathrm{Cok}(\hat{\dirac}^{+}_{(\hat{A},\hat{\Phi})})\cap L^2 = \mathrm{Ker}(\hat{\dirac}^{-}_{(\hat{A},\hat{\Phi})})\cap L^2$ and $\mathrm{Cok}(\lambda^{+}_1\hat{\dirac}^{+}_{(\hat{A},\hat{\Phi})}\lambda^{+}_2)\cap L^2 = \mathrm{Ker}(\lambda^{-}_1\hat{\dirac}^{-}_{(\hat{A},\hat{\Phi})}\lambda^{-}_2)\cap L^2 = (\lambda^{-}_2)^{-1}\cdot(\mathrm{Ker}(\hat{\dirac}^{-}_{(\hat{A},\hat{\Phi})})\cap L^2)$.
				Therefore we obtain $\mathrm{Ind}(\hat{\dirac}^{+}_{(\hat{A},\hat{\Phi})}) = \mathrm{Ind}(\lambda^{+}_1\hat{\dirac}^{+}_{(\hat{A},\hat{\Phi})}\lambda^{+}_2) = \mathrm{Ind}(D^{+})$,
				which completes the proof.
			\end{proof}
		
			By following \cite{Ref:Ati-Sin},
			we calculate the $S^1$-equivariant index $\mathrm{Ind}(\dirac^{\pm}_{A_4})^{S^1}$.
			\begin{Lem}\label{Lem:weight of spinor}
				For $p\in {Z}_4$\;(resp. $Z'$),
				the weights of the fiber $S^{+}_{\tilde{P}}|_{p}$ and $S^{-}_{\tilde{P}}|_{p}$ are $(0,0)$ and $(-1,1)$\;(resp. $(-1,1)$ and $(0,0)$) respectively.
			\end{Lem}
			\begin{proof}
				For $p\in {Z}_4$,
				the projection $\pi|_{B(p,\eps)}:B(p,\eps) \to \pi(B(p,\eps))$ can be identified with the Hopf fibration $\mathbb{R}^4 = \mathbb{C}^2 \to \mathbb{R}^3$ in Section 1.
				By the natural isomorphisms $S^{+}_{\mathbb{C}^2} \simeq \Omega^{0,0}_{\mathbb{C}^2}\oplus\Omega^{0,2}_{\mathbb{C}^2}$ and $S^{-}_{\mathbb{C}^2} \simeq \Omega^{0,1}_{\mathbb{C}^2}$,
				the weights of $S^{+}_{\tilde{P}}|_{p}$ and $S^{-}_{\tilde{P}}|_{p}$ are $(0,0)$ and $(-1,1)$ respectively.
				As a similar way,
				for $p'\in {Z}'_4$,
				the projection $\pi|_{B(p',\eps)}:B(p',\eps) \to \pi(B(p',\eps))$ can be identified with the inverse-oriented Hopf fibration $-\mathbb{R}^4 \to \mathbb{R}^3$.
				Therefore the weights of $S^{+}_{\tilde{P}}|_{p'}$ and $S^{-}_{\tilde{P}}|_{p'}$ are $(-1,1)$ and $(0,0)$ respectively.
			\end{proof}
			\begin{Prp}\label{Prp:S^1-inv ind. formula}
				The $S^1$-invariant index $\mathrm{Ind}(\dirac^{\pm}_{A_4})^{S^1}$ is given as
				\[
					\mathrm{Ind}(\dirac^{\pm}_{A_4})^{S^1} = \mp\sum_{p\in Z} \sum_{k_{p,i}>0}k_{p,i},
				\]
				where $\vec{k}_p =(k_{p,i})\in \mathbb{Z}^r$ is the weight of the monopole $(V,h,A,\Phi)$ at $p\in Z$.
			\end{Prp}
			\begin{proof}
				According to \cite{Ref:Ati-Sin},
				The $S^1$-invariant index $\mathrm{Ind}(\dirac^{\pm}_{A_4})^{S^1}$ is given as
				\[
				\mathrm{Ind}(\dirac^{\pm}_{A_4})^{S^1} = 
				(2\pi)^{-1}\int_{S^1} \sum_{p\in \tilde{Z}_4}\frac{\mathrm{tr}_{\theta}((S^{\pm}_{\tilde{P}}\otimes V_4)|_{p})-\mathrm{tr}_{\theta}((S^{\mp}_{\tilde{P}}\otimes V_4)|_{p})}{\mathrm{tr}_{\theta}(\bigwedge^{-1}T_{p}\tilde{P})}d\theta,
				\]
				where $\mathrm{tr}_{\theta}$ is trace of the action of $\theta\in S^1$ and $\bigwedge^{-1}T_{p}\tilde{P}$ means the virtual vector space $\bigoplus_{i=0}^{\infty}(-1)^i\bigwedge^{i}T_{p}\tilde{P}$.
				Then by Lemma \ref{Lem:weight of spinor} we have 
				\begin{alignat*}{2}
					&\mathrm{tr}_{\theta}((S^{\pm}_{\tilde{P}}\otimes V_4)|_{p})-\mathrm{tr}_{\theta}((S^{\mp}_{\tilde{P}}\otimes V_4)|_{p}) =\pm 2(1+\cos\theta)\sum_{i}\exp(2\pi\sqrt{-1}k_{p,i}\theta)\;&(p\in Z)\\
					&\mathrm{tr}_{\theta}((S^{\pm}_{\tilde{P}}\otimes V_4)|_{p'})-\mathrm{tr}_{\theta}((S^{\mp}_{\tilde{P}}\otimes V_4)|_{p'}) =\mp2r(1-\cos\theta)\;&(p'\in Z')\\
					&\mathrm{tr}_{\theta}(\bigwedge{}^{-1}\,T_{\tilde{p}}\tilde{P}) =4(1-\cos\theta)^2\;&(\tilde{p}\in \tilde{Z}).
				\end{alignat*}
				Hence by straightforward computation we obtain the conclusion.
			\end{proof}
		
			Hence we obtain the following corollary.
			\begin{Cor}\label{Cor:Index formula of monopole}
				The indices of the Dirac operators $\dirac^{\pm}_{(A,\Phi)}$ are given as follows:
				\[
					\mathrm{Ind}(\dirac^{\pm}_{(A,\Phi)})  = \mp\sum_{p\in Z} \sum_{k_{p,i}>0}k_{p,i} = \pm\sum_{p\in Z} \sum_{k_{p,i}<0}k_{p,i},
				\]
					where $\vec{k}_p =(k_{p,i})\in \mathbb{Z}^r$ is the weight of the monopole $(V,h,A,\Phi)$ at $p\in Z$.
			\end{Cor}
	\section{An index formula of Dirac operators on complete $3$-folds}
	Let $(X,g)$ be a complete oriented Riemannian $3$-fold such that the scalar curvature $Sc(g)$ is bounded.
	We fix a spin structure on $X$.
	Let $i: Y \hookrightarrow X$ be a relative compact region with a smooth boundary $\partial Y$.
	We take the orientation of $\partial Y$ to satisfy that $\nu\wedge\mathrm{vol}_{\partial Y}$ is positive for the inward normal unit $1$-form $\nu \in i^{\ast}\Omega^1(X)$.
	\subsection{The non-singular case}
		Following \cite{Ref:Rad},
		we recall the non-singular case.
		Let $(V,h,A)$ be a Hermitian bundle with a connection on $X$ and $\Phi$ be a skew-Hermitian endomorphism on $V$.
		We assume the following conditions.
		\begin{itemize}
			\item 
				Both $\Phi$ and $F(A)$ are bounded.
			\item
				We have $\nabla_A(\Phi)|_{x} = o(1)$ as $x\to\infty$.
			\item
				The inequality $\mathrm{inf}_{x\in X\setminus Y}\bigl\{|\lambda| \mid \mbox{$\lambda$ is an eigenvalue of $\Phi(x)$}\bigr\} > 0$ is satisfied.
		\end{itemize}
		We call this conditions \textit{the R\ra{a}de condition}.
		In \cite{Ref:Rad},
		R\ra{a}de proved the following theorem.
		\begin{Thm}\label{Thm:Rade's formula}
			The differential operators $\dirac^{\pm}_{(A,\Phi)} = \dirac_{A} \pm \Phi:L^2(X,V\otimes S_X) \to L^2(X,V\otimes S_X)$ are closed Fredholm,
			and their indices are given as follows:
			\[
				\mathrm{Ind}(\dirac^{\pm}_{(A,\Phi)})
				= \mp \int_{\partial Y} \mathrm{ch}(V^{+})
				= \pm \int_{\partial Y} \mathrm{ch}(V^{-}),
			\]
			where $V^{\pm}$ is a subbundle of $V|_{\partial Y}$ spanned by the eigenvectors of $\mp\sqrt{-1}\Phi$ with positive eigenvalues.
		\end{Thm}
	\subsection{The indices of twisted flat Dirac monopoles}
		Let $(L_k,h_k,A_k,\Phi_k)$ be the flat Dirac monopole of weight $k\in\mathbb{Z}$.
		For $a\in\mathbb{R}\setminus\{0\}$,
		we set $\Phi_{a,k} := \sqrt{-1}\left(a + (k/2r_3)\right)$.
		Then $(L_k,h_k,A_k,\Phi_{a,k})$ is also a Dirac-type monopole on $(\mathbb{R}^3,\{0\})$.
		\begin{Prp}
			The operators $\dirac^{\pm}_{(A_k,\Phi_{a,k})}$ are Fredholm.
			Moreover,
			we have $\mathrm{Ind}(\dirac^{\pm}_{(A_k,\Phi_{a,k})})=0$ if $ak>0$. 
		\end{Prp}
		\begin{proof}
			The Fredholmness of $\dirac^{\pm}_{(A_k,\Phi_{a,k})}$ follows from Corollary \ref{Cor:Relich type thm for Dirac monopole} and some standard arguments.
			The proof for the case $a>0$ and $k>0$ works for the case $a<0$ and $k<0$ mutatis mutandis.
			Hence we may assume $a>0$ and $k>0$.
			We take $f^{\pm}_{\nu} \in L^2(S^2,S^{\pm}_{S^2}\otimes \mathcal{O}(k))\;(\nu\in\mathbb{N})$ and $n_{\nu}>0$ as in subsection \ref{subsec:Loc. prop. of flat Dirac mon.}.
			We set vector subspaces $W_0:=H^0(\mathbb{P}^1,\mathcal{O}(k-1))\times \{0\}$ and $W_{\nu}:=(f^{+}_{\nu},0)\mathbb{C}\oplus (0,f^{-}_{\nu})\mathbb{C}$ of $L^2(S^2,\mathcal{O}(k)\otimes S_{S^2}) = L^2(S^2,\mathcal{O}(k)\otimes S^{+}_{S^2})\times L^2(S^2,\mathcal{O}(k)\otimes S^{-}_{S^2})$.
			Then we have a decomposition $L^2(S^2,\mathcal{O}(k)\otimes S_{S^2}) = \widehat{\bigoplus}_{\nu\geq0}W_{\nu}$,
			where $\widehat{\bigoplus}$ means $L^2$-completion of the direct sum.
			Hence we obtain the decomposition $L^2(\mathbb{R}^3\setminus\{0\}, L_k\otimes S_{\mathbb{R}^3}) = \widehat{\bigoplus}_{\nu\geq0} W_{\nu}\otimes L^2(\mathbb{R}_{>0},r^2dr)$,
			where $L^2(\mathbb{R}_{>0},r^2dr)$ is the weighted $L^2$-space on $\mathbb{R}_{>0}$ with the norm $||f||^2 = \int_{\mathbb{R}_{>0}} r^2 |f(r)|^2 dr$.
			We denote by $E_{\nu}$ the space $W_{\nu}\otimes L^2(\mathbb{R}_{>0},r^2dr)$.
			The Dirac operators $\dirac^{\pm}_{(A_k,\Phi_{a,k})}$ preserves this decomposition,
			and hence we obtain $\mathrm{Ind}(\dirac^{\pm}_{(A_k,\Phi_{a,k})})=\sum_{\nu}\mathrm{Ind}\left(\dirac^{\pm}_{(A_k,\Phi_{a,k})}|_{E_{\nu}}:E_{\nu}\to E_{\nu}\right)$.
			Here we prepare the following lemma.
			\begin{Lem}\label{Lem:ind. of ode.}
				We take Hermitian matrices
				\[
				A_{\nu} = 
				\left(
				\begin{array}{ll}
				-(2+k)/2&\sqrt{-1}n_{\nu}\\
				-\sqrt{-1}n_{\nu}&-(2-k)/2
				\end{array}
				\right),\;\;\;\;
				B = 
				\left(
				\begin{array}{ll}
				a&0\\
				0&-a
				\end{array}
				\right).
				\]
				We set the closed operator $P_{\nu}:\mathbb{C}^2\otimes L^2(\mathbb{R}_{>0},r^2dr) \to \mathbb{C}^2\otimes L^2(\mathbb{R}_{>0},r^2dr)$ to be $P_{\nu}(v) := \partial_r v -( A_{\nu}v/r + Bv)$.
				Then $P_{\nu}$ is closed Fredholm and $\mathrm{Ind}(P_{\nu})=0$.
			\end{Lem} 
			By this lemma we have $\mathrm{Ind}\left(\dirac^{\pm}_{(A_k,\Phi_{a,k})}|_{E_{\nu}}\right)=0$ unless $i=0$. Hence $\mathrm{Ind}(\dirac^{\pm}_{(A_k,\Phi_{a,k})})=\mathrm{Ind}\left(\dirac^{\pm}_{(A_k,\Phi_{a,k})}|_{E_{0}} \right)$.
			Moreover, we obtain $\mathrm{Ind}\left(\dirac^{\pm}_{(A_k,\Phi_{a,k})}|_{E_0}\right)=0$ by a straight calculation.
		\end{proof}
		
		\begin{proof}[{\upshape \textbf{(proof of Lemma \ref{Lem:ind. of ode.})}}]
			The Fredholmness can be easily seen.
			We take a function $C_{\nu}:\mathbb{R}_{>0}\to \mathrm{Mat}(2,\mathbb{C})$ as
			\[
				C_{\nu}(r):=
				\left\{
					\begin{array}{ll}
						A_{\nu}/r&(r\leq1)\\
						B&(r>1)
					\end{array}
				\right.
			\]
			and set a differential operator $\tilde{P}_{\nu}$ to be $\tilde{P}(v) := \partial_r v - C_{\nu}(r)v$.
			Since a compact perturbation does not change the index,
			$P_{\nu}$ and $\tilde{P}_{\nu}$ have the same indices.
			We can write any elements of the kernels of $\tilde{P}_{\nu}$ and the adjoint operator $\tilde{P}_{\nu}^{\bigstar}$ explicitly,
			and there are no non-zero $L^2$-solutions of $\tilde{P}_{\nu}(v)=0$ and $\tilde{P}_{\nu}^{\bigstar}(v)=0$.
			Hence we obtain $\mathrm{Ind}(P_{\nu})=0$,
			which is the desired equality.
		\end{proof}
	\subsection{The general case}
		Let $Z \subset Y\setminus \partial Y$ be a finite subset.
		Let $(V,h,A,\Phi)$ be a Dirac-type singular monopole on $(X,Z)$ of rank $r$ which satisfies the R\ra{a}de condition.
		We denote by $\vec{k}_{p}=(k_{p,i})\in\mathbb{Z}^r$ the weight of $(V,h,A,\Phi)$ at $p\in Z$.
		\begin{Thm}\label{Thm:Index formula for general complete $3$-folds}
			The Dirac operators $\dirac^{\pm}_{(A,\Phi)}$ are Fredholm and adjoint each other.
			The indices of $\dirac^{\pm}_{(A,\Phi)}$ are given as follows:
			\[
				\mathrm{Ind}(\dirac^{\pm}_{(A,\Phi)})
				=
				\mp\left\{
					\sum_{p\in Z}\sum_{k_{p,i}>0} k_{p,i}
					+ \int_{\partial Y}ch(V^{+})
				\right\}
				=
				\pm\left\{
					\sum_{p\in Z}\sum_{k_{p,i}<0} k_{p,i}
					+ \int_{\partial Y}ch(V^{-})
				\right\}.
			\]
		\end{Thm}
		\begin{proof}
			We may assume that $X$ is connected.
			The former claims are easy consequences of Corollary \ref{Cor:Relich type thm for Dirac monopole} and results in \cite{Ref:Rad}.
			We calculate the indices of $\dirac^{\pm}_{(A,\Phi)}$ by using the excision formula in \cite[Appendix B]{Ref:Cha1}.
			We set $k:= \sum_{p\in Z}\sum_{i}k_{p,i}$.
			
			First we consider the case $k=0$.
			Let $(V_0,h_0,A_0)$ be a trivial Hermitian bundle of rank $r$ with the trivial connection on $S^3$ and $\Phi_0$ the zero endomorphism on $V_0$.
			Let $U_N$ be the northern closed hall ball of $S^3$.
			We take a compact neighborhood $U$ of $Z$ that is diffeomorphic to a closed ball.
			We replace $(V,h,A,\Phi)|_U$ and $(V_0,h_0,A_0,\Phi_0)|_{U_N}$,
			and obtain $(\tilde{V}^{0},\tilde{h}^{0},\tilde{A}^{0},\tilde{\Phi}^{0})$ on $X$ and $(\tilde{V}_0,\tilde{h}_0,\tilde{A}_0,\tilde{\Phi}_0)$ on $S^3$.
			Then by the excision formula we have $\mathrm{ind}(\dirac^{\pm}_{(A,\Phi)}) + \mathrm{ind}(\dirac^{\pm}_{(A_0,\Phi_0)}) =\mathrm{ind}(\dirac^{\pm}_{(\tilde{A}^0,\tilde{\Phi}^0)}) + \mathrm{ind}(\dirac^{\pm}_{(\tilde{A}_0,\tilde{\Phi}_0)})$.
			Hence we obtain $\mathrm{ind}(\dirac^{\pm}_{(A,\Phi)}) = \mp\left(\sum_{k_{p,i}>0}k_{p,i} + \int_{\partial Y} \mathrm{ch}(V^{+})\right)$ by Corollary \ref{Cor:Index formula of monopole} and Theorem \ref{Thm:Rade's formula}.
			
			Next we consider the case $k\neq0$.
			The proof for the case $k>0$ remain valid for $k<0$ mutatis mutandis.
			Hence we may assume $k>0$. 
			Let $(V_1,h_1,A_1)$ be a Hermitian bundle of rank $r$ with a connection on $S^3$ outside the north pole $p_N$ and the south pole $p_S$, and $\Phi_1$ be a skew-Hermitian endomorphism of $V_1$.
			We assume that $(V_1,h_1,A_1,\Phi_1)$ is a Dirac-type singular monopole of weight $(k,0,\dots,0)$\;(resp. $(-k,0,\dots,0)$) on a neighborhood of $p_N$\;(resp. $p_S$).
			We replace $(V,h,A,\Phi)|_U$ and $(V_1,h_1,A_1,\Phi_1)|_{U_N}$,
			and obtain $(\tilde{V}^1,\tilde{h}^1,\tilde{A}^1,\tilde{\Phi}^1)$ on $X$ and $(\tilde{V}_1,\tilde{h}_1,\tilde{A}_1,\tilde{\Phi})$ on $S^3$.
			Then the excision formula shows $\mathrm{ind}(\dirac^{\pm}_{(A,\Phi)}) + \mathrm{ind}(\dirac^{\pm}_{(A_1,\Phi_1)}) = \mathrm{ind}(\dirac^{\pm}_{(\tilde{A}^1,\tilde{\Phi}^1)}) + \mathrm{ind}(\dirac^{\pm}_{(\tilde{A}_1,\tilde{\Phi}_1)})$.
			Hence  $\mathrm{ind}(\dirac^{\pm}_{(A,\Phi)}) \mp k = \mathrm{ind}(\dirac^{\pm}_{(\tilde{A}^1,\tilde{\Phi}^1)}) \mp \sum_{k_{p,i}>0}k_{p,i}$.
			We set $(V_2,h_2,A_2,\Phi_2):= (L_{-k},h_{-k},A_{-k},\Phi_{-k,-1})\oplus (\underline{\mathbb{C}^{r-1}}, \underline{h_{\mathbb{C}^{r-1}}},\underline{d},0)$ on $\mathbb{R}^3$,
			where $(\underline{\mathbb{C}^{r-1}},\underline{h_{\mathbb{C}^{r-1}}},\underline{d})$ be a trivial Hermitian bundle with the trivial connection on $\mathbb{R}^3$.
			We denote by $p\in X$ the singular point of $(\tilde{V}^1,\tilde{h}^1,\tilde{A}^1,\tilde{\Phi}^1)$.
			We glue $(\tilde{V}^1,\tilde{h}^1,\tilde{A}^1,\tilde{\Phi}^1)|_{X \setminus B(p,\eps)}$ and $(V_2,h_2,A_2,\Phi_2)|_{\mathbb{R}^3\setminus B(0,\eps)}$,
			and obtain $(\tilde{V}^2,\tilde{h}^2,\tilde{A}^2,\tilde{\Phi}^2)$ on $\tilde{X} = \left((X \setminus B(p,\eps))\sqcup \mathbb{R}^3\setminus B(0,\eps)\right)/\sim$,
			where $\sim$ is an identification of their boundaries.
			We also glue $ (\tilde{V}^1,\tilde{h}^1,\tilde{A}^1,\tilde{\Phi}^1)|_{\overline{B(p,\eps)}}$ and $(V_2,h_2,A_2,\Phi_2)|_{\overline{B(0,\eps)}}$ and obtain $(\tilde{V}_2,\tilde{h}_2,\tilde{A}_2,\tilde{\Phi}_2)$ on $S^3_{\eps} := (\overline{B(p,\eps)}\sqcup \overline{B(0,\eps)})/\sim$,
			where over-line means the closure.
			Then by the excision formula we have 
			$
				\mathrm{ind}(\dirac^{\pm}_{(\tilde{A}^1,\tilde{\Phi}^1)})
				=
				\mathrm{ind}(\dirac^{\pm}_{(\tilde{A}^1,\tilde{\Phi}^1)}) + \mathrm{ind}(\dirac^{\pm}_{(A_2,\Phi_2)})
				=
				\mathrm{ind}(\dirac^{\pm}_{(\tilde{A}^2,\tilde{\Phi}^2)}) + \mathrm{ind}(\dirac^{\pm}_{(\tilde{A}_2,\tilde{\Phi}_2)})
				=
				\mathrm{ind}(\dirac^{\pm}_{(\tilde{A}^2,\tilde{\Phi}^2)}) \mp
				k
			$.
			Since the tuple $(\tilde{V}^2,\tilde{h}^2,\tilde{A}^2,\tilde{\Phi}^2)$ satisfies the R\ra{a}de condition,
			we obtain $\mathrm{ind}(\dirac^{\pm}_{(\tilde{A}^2,\tilde{\Phi}^2)}) = \mp\int_{\partial Y} ch(V^{+})$.
			As a consequence of the above arguments,
			we obtain $\mathrm{ind}(\dirac^{\pm}_{(A,\Phi)}) = \mp(\int_{\partial Y} ch(V^{+}) + \sum_{k_{p,i}>0}k_{p,i})$,
			which is the desired equation.
		\end{proof}

\end{document}